\documentclass[12pt,reqno]{amsart}

\usepackage[T1]{fontenc}
\pdfoutput=1
\usepackage{mathptmx,microtype,hyperref,graphicx,dsfont}
\usepackage{amsthm,amsmath,amssymb}
\usepackage[foot]{amsaddr}
\usepackage{enumitem}
\usepackage[nameinlink,capitalise]{cleveref}
\usepackage[font=footnotesize,margin=1.6cm]{caption}
\usepackage{cite}
\usepackage{framed}
\usepackage{xcolor}

\usepackage{todonotes}

\crefname{theorem}{Theorem}{Theorems}
\crefname{thm}{Theorem}{Theorems}
\crefname{mainthm}{Theorem}{Theorems}
\crefname{lemma}{Lemma}{Lemmas}
\crefname{lem}{Lemma}{Lemmas}
\crefname{remark}{Remark}{Remarks}
\crefname{claim}{Claim}{Claims}
\crefname{subclaim}{Sub-claim}{Sub-claims}
\crefname{prop}{Proposition}{Propositions}
\crefname{proposition}{Proposition}{Propositions}
\crefname{defn}{Definition}{Definitions}
\crefname{corollary}{Corollary}{Corollaries}
\crefname{conjecture}{Conjecture}{Conjectures}
\crefname{question}{Question}{Questions}
\crefname{chapter}{Chapter}{Chapters}
\crefname{section}{Section}{Sections}
\crefname{figure}{Figure}{Figures}
\newcommand{\crefrangeconjunction}{--}
\theoremstyle{plain}

\newtheorem{thm}{Theorem}
\newtheorem*{thm*}{Theorem}
\newtheorem{lemma}[thm]{Lemma}

\newtheorem{corollary}[thm]{Corollary}
\newtheorem{prop}[thm]{Proposition}

\theoremstyle{definition}

\theoremstyle{remark}

\numberwithin{equation}{section}

\newcommand\cA{\mathcal{A}}
\newcommand{\cB}{{\mathcal B}}

\newcommand\cE{\mathcal{E}}

\newcommand{\cG}{{\mathcal G}}
\newcommand{\cI}{{\mathcal I}}
\newcommand\cJ{\mathcal{J}}

\newcommand\cP{\mathcal{P}}
\newcommand\cR{\mathcal{R}}
\newcommand\cS{\mathcal{S}}

\newcommand{\cV}{{\mathcal V}}
\newcommand{\cW}{{\mathcal W}}

\newcommand\cZ{\mathcal{Z}}

\newcommand\bS{{\bf S}}
\newcommand\bW{{\bf W}}
\newcommand\bY{{\bf Y}}

\newcommand{\bfs}{{\bf s}}
\newcommand{\bfv}{{\bf v}}

\newcommand{\bfY}{{\bf Y}}
\newcommand{\bfA}{{\bf A}}

\newcommand\Z{\mathbb{Z}}
\newcommand\R{\mathbb{R}}

\newcommand\bigmid{\mathrel{\Big|}}
\newcommand\sgn{{\rm sgn}}

\newcommand\Bern{\rm{Bernoulli}}

\renewcommand{\P}{{\mathbb P}}
\newcommand\E{{\mathbb E}}
\newcommand{\1}{{\bf 1}}

\newcommand\all{\:\forall}

\newcommand\sY{\check Y}
\newcommand\sA{\check A}

\newcommand{\eps}{\varepsilon}
\renewcommand{\le}{\leqslant}
\renewcommand{\ge}{\geqslant}

\makeatletter
\def\@tocline#1#2#3#4#5#6#7{\relax
  \ifnum #1>\c@tocdepth 
  \else
    \par \addpenalty\@secpenalty\addvspace{#2}%
    \begingroup \hyphenpenalty\@M
    \@ifempty{#4}{%
      \@tempdima\csname r@tocindent\number#1\endcsname\relax
    }{%
      \@tempdima#4\relax
    }%
    \parindent\z@ \leftskip#3\relax \advance\leftskip\@tempdima\relax
    \rightskip\@pnumwidth plus4em \parfillskip-\@pnumwidth
    #5\leavevmode\hskip-\@tempdima
      \ifcase #1
       \or\or \hskip 1em \or \hskip 2em \else \hskip 3em \fi%
      #6\nobreak\relax
    \dotfill\hbox to\@pnumwidth{\@tocpagenum{#7}}\par
    \nobreak
    \endgroup
  \fi}
\makeatother

\newcommand\rw{Y} 
\newcommand\rwa{A} 
\newcommand\nex{{N_n}} 
\newcommand\nexa{{M_n}} 

\author[S. Donderwinkel]{Serte Donderwinkel}
\address{University of Groningen, Bernoulli Institute for Mathematics, Computer Science and AI, and CogniGron (Groningen Cognitive Systems and Materials Center)}
\email{s.a.donderwinkel@rug.nl}

\author[B. Kolesnik]{Brett Kolesnik}
\address{University of Oxford, Department of Statistics and Magdalen College}
\email{brett.kolesnik@stats.ox.ac.uk}


\begin{document}

\title[Tournaments and random walks]
{Tournaments and random walks}

\begin{abstract} 
We study the relationship between
tournaments and random walks. 
This connection was first observed by 
Erd{\H{o}}s and Moser. 
Winston and Kleitman came close to 
showing that $S_n=\Theta(4^n/n^{5/2})$. 
Building on this,
and works by  Tak\'acs, 
these asymptotic bounds were confirmed by Kim and Pittel. 

In this work, we verify Moser's conjecture that 
$S_n\sim C4^n/n^{5/2}$, using limit 
theory for 
integrated random walk bridges.
Moreover, we show that 
$C$ can be described 
in terms of random walks. 
Combining this with a recent proof 
and number-theoretic description of $C$ 
by 
the second author, 
we obtain an analogue of 
Louchard's formula, for the Laplace transform of the 
squared Brownian excursion/Airy area measure. 
Finally, we describe the scaling limit 
of random score sequences, in 
terms of the Kolmogorov excursions,    
studied recently by 
B\"{a}r, Duraj and Wachtel.

Our results can also be interpreted as 
answering questions related to 
a class of random polymers, 
which began with influential work of Sina\u{\i}. 
From this point of view,
our methods yield
the precise asymptotics
of a persistence probability, related to the pinning/wetting 
models from statistical physics, that was estimated 
up to constants by 
Aurzada, Dereich and Lifshits, 
as conjectured by 
Caravenna and Deuschel. 

%
%
\end{abstract}

\maketitle

\section{Introduction}\label{S_intro}

A {\it tournament} $T$ is an orientation of the complete graph 
$K_n$ on $[n]=\{1,2,\ldots,n\}$. 
Intuitively, we think of vertices as players and edges $\{i,j\}$ as games, 
oriented as $i\to j$ if $i$ wins against $j$. 
The score sequence ${\bf s}(T)$ of $T$ is the weakly increasing 
rearrangement of the out-degree sequence, which lists the total 
number of wins by each team. 

We call $\bfv_n=(0,1,\ldots,n-1)$ the 
{\it standard score sequence}. 
This corresponds to the acyclic tournament, 
in which $i$ wins against all $j<i$. 
Intuitively, $\bfv_n$ is as ``spread out'' as possible. 
Indeed, Landau \cite{Lan53} proved that 
${\bf s}=(s_1\le\cdots\le s_n)\in\Z^n$ 
is a score sequence if and only if 
${\bf s}$ is \emph{majorized} 
(see, e.g., Marshall, Olkin and Arnold \cite{MOA11}) 
by $\bfv_n$, 
that is, 
\begin{align}
\label{E_maj1}&\sum_{i=1}^k s_i\ge \binom{k}{2},\quad\quad 1\le k<n,\\ 
\label{E_maj2}&\sum_{i=1}^n s_i=\binom{n}{2}. 
\end{align}
These conditions are necessary, since 
any $k$ teams must win at least the 
number of games between them ${k\choose2}$, 
and there are ${n\choose2}$ games in total. 
On the other hand, sufficiency follows directly, e.g., by 
the Havel--Hakimi \cite{Hav55,Hak62}
algorithm and standard majorization techniques. 

Equivalently, score sequences correspond to the weakly increasing lattice
points of the permutahedron $\Pi_{n-1}$, that is, 
the convex hull of $\bfv_n$
and its permutations. The set of all lattice points, on the other hand, 
is in bijection with the set of all spanning forests $F\subset K_n$.
Therefore, score sequences 
correspond to a 
special class of spanning forests. 
See, e.g., Stanley \cite{Sta80}
(cf.\ Postnikov \cite{Pos09}).

\subsection{Counting score sequences}
In contrast, 
the number $S_n$ of score sequences 
appears not to have a simple description. 
For instance, MacMahon \cite{Mac20} 
calculated
$S_n$ for small values of $n$
using symmetric functions, but such an analysis becomes 
challenging very quickly. 
Claesson, Dukes, Frankl\'in and 
Stef\'ansson \cite{CDFS23} recently verified a recursion, 
conjectured by Hanna \cite{A000571}. 
However, this does not lead to  
a simple closed form expression, but a rather complicated  
cycle/product formula \cite[(8)]{CDFS23}
(formally equivalent to the recursion itself). 

The first asymptotic bounds on $S_n$ were obtained 
by Erd{\H{o}}s and Moser 
(see, e.g., Moon's \cite{Moo68} classic monograph), 
who, according to Kleitman \cite[p.\ 209]{Kle69}, 
conjectured that 
$S_n=\Theta(4^n/n^{5/2})$.
Furthermore, 
in 1968 symposium proceedings, 
Moser \cite[p.\ 165]{Mos71} stated that 
\begin{quote}
We feel that 
the asymptotic behavior should be of the form 
$c4^n/n^{5/2}$ and I hereby offer \$25.00 to anyone who will 
find a proof of this conjecture.
\end{quote}
Winston and Kleitman \cite{WK83} came close 
to proving that $S_n=\Theta(4^n/n^{5/2})$, 
up to a certain bound
on the $q$-Catalan numbers (see \cite[\S9--11]{WK83}). 
Building on works by Tak\'acs \cite{Tak86,Tak91}, 
Kim and Pittel \cite{KP00} finally 
verified this to be the correct asymptotic order
of $S_n$.

Our first result is a short, probabilistic proof 
(see \cref{S_Theta}) of this fact,  via its 
connection with 
integrated random walk bridges.

\begin{thm}
[\hspace{1sp}\cite{WK83,KP00}]
\label{T_order}
$S_n=\Theta(4^n/n^{5/2})$. 
\end{thm}

This connection 
was, in fact, first recorded by Moser \cite{Mos71}, 
and it plays a central role in the arguments of 
Winston and Kleitman
(cf.\ 
\cite[p.\ 212]{WK83}). 

If a score sequence $\bfs=(s_1,\ldots,s_n)$ is drawn as a 
``bar graph,'' we obtain an  
$\uparrow,\rightarrow$ lattice walk $\cW(\bfs)$
from $(0,0)$ to $(n,n)$. 
The $\rightarrow$ steps are taken at times $s_k+k$, 
at which point the $k$th ``bar'' of height $s_k$ is completed. 
We let $\cW_n=\cW(\bfv_n)$ denote the {\it staircase walk}
corresponding to the standard score sequence $\bfv_n$. 
See \cref{F_WK}. 

Rotating a walk $\cW(\bfs)$ clockwise by $\pi/4$, 
we obtain a bridge path $\cP(\bfs)$. 
More precisely, this path has increments equal to $-1$
at times $t=s_k+k$, $1\le k\le n$, and equal to $+1$ otherwise. 
Let $\check\cP_n=\cP(\bfv_n)$ denote the {\it sawtooth path,}
obtained by rotating $\cW_n$. Note that 
$\check\cP_n$
oscillates between $0$ and $-1$. 
The conditions 
\eqref{E_maj1} and \eqref{E_maj2}
for score sequences $\bfs$ 
correspond to $\cP(\bfs)$ having 
partial areas $\ge0$
and 
total area $0$
above  $\check\cP_n$. 
See \cref{F_WKrot}.

The program outlined in \cite{WK83} 
ignores the partial areas condition \eqref{E_maj2},
which is the most technically challenging. 
In this work, we will take both conditions 
\eqref{E_maj1} and \eqref{E_maj2}
into account, taking inspiration from the 
recent work of 
Balister, the first author, 
Groenland, Johnston and Scott \cite{BDGJS22}, 
which asymptotically enumerates 
the number of graphical sequences. 
The connection with random walks \cite{Mos71,WK83}
continues to hold in this context. 
See \cref{S_GraphicalSeq} below for
more details. 

Our next result identifies the precise asymptotics of $S_n$, 
and proves Moser's \cite{Mos71} conjecture.

\begin{thm}
\label{T_C}
As $n\to\infty$, we have that 
$n^{5/2}S_n/4^n\to C$. 
\end{thm}

This result follows by 
\cref{P_KeyF,P_numer,P_MntoM,T_asySn_B}
below.

\subsection{The constant $C$}
The constant $C$ is described in \cref{S_C} below, in terms of 
simple random walk. 
\cref{T_C} was recently proved
by the second author in \cite{Kol23}, using
limit theory for 
infinitely divisible distributions, together with the recursion in \cite{CDFS23}, 
mentioned above.  
In this context, $C$ has
an alternative  
description in terms of the Erd{\H{o}}s--Ginzburg--Ziv numbers 
from additive number theory, see \cref{S_C} below.

\begin{figure}[h]
\centering
\includegraphics[scale=1]{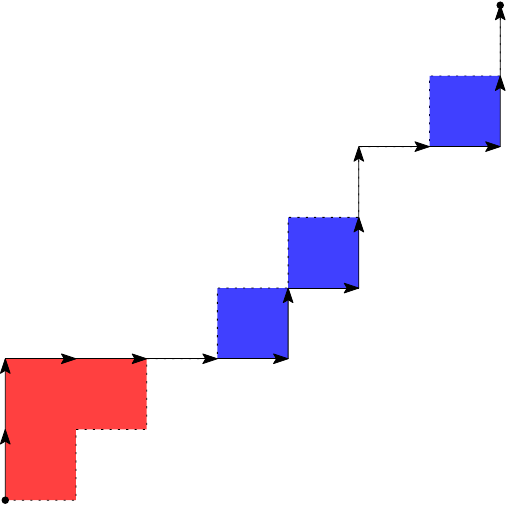}
\caption{The $\uparrow,\rightarrow$ lattice walk $\cW(\bfs)$
from $(0,0)$ to $(7,7)$ 
corresponding to 
score sequence $\bfs=2222355$. 
Positive/negative excursions
above/below the staircase walk $\cW_7$ (dotted), corresponding to 
standard score sequence 
$\bfv_7=0123456$, are red/blue. 
}
\label{F_WK}
\end{figure}

\begin{figure}[h]
\centering
\includegraphics[scale=1]{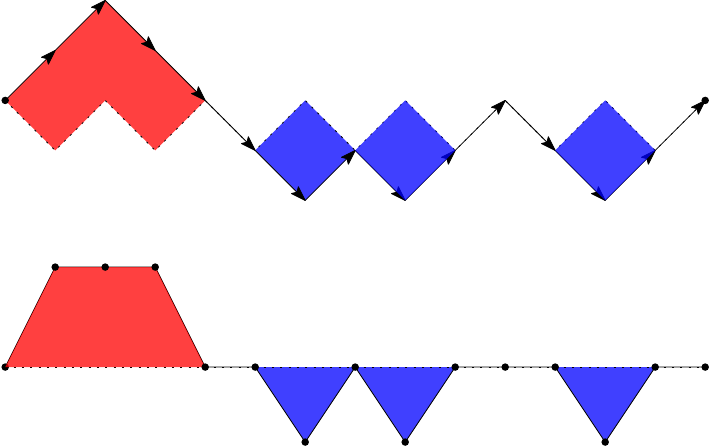}
\caption{
{\it Above:} The bridge path $\cP(\bfs)$ 
obtained by rotating $\cW(\bfs)$
in \cref{F_WK} and the sawtooth path $\check \cP_7$ (dotted)
obtained by rotating the staircase walk $\cW_7$. 
{\it Below:} The height process above $\check \cP_7$, 
with positive/negative contributions to the area shaded
in red/blue. 
}
\label{F_WKrot}
\end{figure}

\subsection{An Airy integral}
Equating the two expressions for $C$, 
given by \cref{T_asySn_B,T_asySn_N} below, we obtain the following 
equation, which is reminiscent of 
Louchard's formula  
\cite[p.\ 490]{Lou84} (cf.\ 
\cite[\S13]{Jan07})
for the integrated Laplace transform of the {\it Brownian excursion area} $\cB$
in terms of the Airy function. We recall that 
$\cB$ is the area under the path of Brownian motion, 
conditioned to stay positive and return to 0
at time 1. 
Alternatively, 
$\cA = 2^{3/2}\cB$ is the {\it Airy area measure}.

\begin{corollary}\label{C_AiryInt}
The Brownian excursion area $\cB$ satisfies 
\begin{equation}\label{E_AiryInt}
\int_0^\infty
\E[\exp(-6x^3\cB^2)]
\sqrt{1+1/x}\: dx
=\frac{\pi}{\sqrt3}.
\end{equation}
\end{corollary}

The simplicity of the right hand side calls for a
more direct explanation, perhaps by expanding the left hand side, 
integrating term by term, and using known facts about the 
fractional, negative moments of $\cB$.

\subsection{Random polymers}
A straightforward adaptation of 
our proof of \cref{T_C}
also leads to 
a sharpening of the main result in  
Aurzada, Dereich, and Lifshits  \cite{ADL14}.

Consider a simple random walk 
$Y_n$ started at $Y_0=0$, 
and let $A_n=\sum_{i=1}^n Y_i$ be its 
area process. In \cite[Theorem 1.1]{ADL14}, it is shown that 
\begin{equation}\label{E_phi_n}
\phi_n=\P(A_1,\ldots,A_{4n}\ge0\mid Y_{4n}=A_{4n}=0)
=\Theta(n^{-1/2}),
\end{equation}
as conjectured by 
Caravenna and Deuschel \cite[\S1.5]{CD08}, 
in their study of the {\it pinning/wetting models}.

As discussed in \cite[p.\ 2]{ADL14}, 
the conditions in \eqref{E_phi_n} are meant to 
``model a polymer chain with Laplace interaction and zero boundary conditions.''
The history of such questions began with 
the influential work of Sina\u{\i} \cite{Sin92}, who showed that 
the {\it persistence probability} 
$\P(A_1,\ldots,A_{n}\ge0)=\Theta(n^{-1/4})$. 
See, e.g., the survey by 
Aurzada and Simon \cite{AS15}
for more on persistence probabilities and their applications. 

Our 
proof of \cref{T_C} can be adapted to 
obtain the following.

\begin{corollary}\label{cor:polymers}
As $n\to\infty$, 
we have that 
$n^{1/2} \phi_n 
\to C'$. 
\end{corollary}

The definition of $C'$ is similar to the definition of $C$. 
See \cref{sec:polymer} below.

In fact, the scheme in \cref{S_Theta} 
can also be used to give a simpler proof of 
\eqref{E_phi_n} itself.

We will give an alternative, combinatorial proof of this result,
along with an explicit, number-theoretic description of $C'$
in \cite{DK24b}.

\subsection{Scaling random score sequences}
Finally, we investigate the ``shape''
of random score sequences. 
{\it Kolmogorov excursions}
$(\bfY_t,\bfA_t)$ were recently studied by 
B\"{a}r, Duraj and Wachtel \cite{BDW23}.
Such a process is, informally, obtained  
(using $h$-transforms) by 
considering a Brownian bridge $\bfY_t$, 
conditioned on its area process $\bfA_t$
staying $\ge0$ and returning to $0$ at time $1$.

\begin{thm}
\label{T_ScalingLimit}
Let
${\bf S}^{(n)}=({\bf S}^{(n)}_1\le \cdots\le {\bf S}^{(n)}_n)$ 
be a uniformly random score 
sequence of length $n$. 
Then, as $n\to\infty$, we have that
\begin{equation}\label{E_ScalingLimit}
\left(\frac{{\bf S}^{(n)}_{\lfloor t n \rfloor}-\lfloor t n \rfloor}{\sqrt{n}},\: 
0\le  t\le  1\right)
\overset{d}{\to}
\left(\bfY_t,\: 0\le  t\le  1\right), 
\end{equation}
where $(\bfY_t,\bfA_t)$ is a Kolmogorov excursion. 
\end{thm}

Therefore, a random score sequence ${\bf S}^{(n)}$ has 
Brownian $O(\sqrt{n})$ fluctuations around its mean $\bfv_n$
(the staircase walk $\lfloor t n \rfloor$). 

An analogous fact for uniformly 
random tournaments ${\bf T}^{(n)}$
was proved by Spencer \cite{Spe71,Spe80}
(cf.\ Bollob\'{a}s and Scott \cite{BS15} and references therein). 
That is, the score sequence $\bfs({\bf T}^{(n)})$
also has Brownian $O(\sqrt{n})$ fluctuations around its mean, in this case  
${\bf u}_n=(n-1,\ldots,n-1)/2$. 

We note that 
$\bfv_n$ and ${\bf u}_n$
are 
the most extreme score sequences, 
being  the most and least ``spread out'' as possible, respectively.
The additional entropy, associated with weighting each $\bfs'$ according to the number 
of tournaments $T$ such that $\bfs(T)=\bfs'$, shifts us 
from one extreme to the other. 

From a geometric point of view, 
\cref{T_ScalingLimit} implies that 
most weakly increasing lattice points
in the permutahedron 
are near the vertex $\bfv_n$.

\subsection{Acknowledgments}
We thank 
Sergi Elizalde, 
Christina Goldschmidt,
Svante Janson
and 
Jim Pitman
for helpful discussions. 
SD would like to acknowledge the financial support 
of the CogniGron research center
and the Ubbo Emmius Funds (Univ.\ of Groningen).


\section{The constant and Airy integral}
\label{S_C}

In this section, we describe two equivalent descriptions of the constant
$C$ in \cref{T_C}. Combining these, we will obtain 
the Airy integral \eqref{C_AiryInt}. 
We will also introduce notation that will
be used throughout this work, and 
formalize the connection between score sequences
and random walks, discussed above.

\subsection{Via random walk bridges}
\label{S_CviaRW}

As discussed above, 
enumerating $S_n$ is equivalent to finding the probability 
that a uniformly random 
$\uparrow,\rightarrow$ lattice walk $\cW$ from $(0,0)$ to $(n,n)$ 
has partial areas $\ge0$ and 
total area $0$
above the staircase walk $\cW_n$. 
Rotating clockwise by $\pi/4$, this can be rephrased in terms of 
a simple (symmetric) random walk
$(Y_k:k\ge0)$ on $\Z$. 

Let 
\begin{equation}\label{E_checkYA}
\check Y_k=Y_k+\1_{\{k\text{ odd}\}},
\quad \quad
\check A_k=\sum_{\ell=1}^k \check Y_\ell
\end{equation}
be the height and partial area at time $k$, 
above the sawtooth path $\check\cP_k$. 
We note that 
$\check A_k=A_k+\lceil k/2\rceil$, where
$ A_k=\sum_{\ell=1}^k  Y_\ell$. 

For ease of notation, throughout this work, we let 
\begin{align}
\cA_n&=\{\check\rwa_1,\dots,\check\rwa_{2n}\ge 0\},\label{E_cA}\\
\cZ_n&=\{\check\rw_{2n}=\check\rwa_{2n}=0\}.\label{E_cZ}
\end{align}
Next, we observe that a random walk 
corresponds to a score sequence
if and only if 
both of these events occur. 

\begin{lemma}
\label{P_SnBridge}
The number $S_n$ of score sequences 
of length $n$ 
satisfies   
\begin{equation}\label{E_ScS}
S_n
=4^n\: \P(\cS_n), 
\end{equation}
where 
\begin{equation}\label{E_cS}
\cS_n=\cA_n\cap \cZ_n.
\end{equation}
\end{lemma}

This relation is immediately clear (see, e.g., \cref{F_WK} above), 
so we omit a formal proof. In fact, \eqref{E_ScS} is essentially a probabilistic
reformulation of the observations in \cite{WK83}, discussed above. 
Indeed, by \eqref{E_maj1} and \eqref{E_maj2},
it can be seen that   
\[
S_n={2n \choose n}\times 
\P(\cA_n, \: \check A_{2n}=0
\mid \check\rw_{2n}=0 ). 
\]

Moreover, 
on the event $\check\rw_{2n}=0$, 
when verifying $\cA_n$, 
it suffices to check only at times $t$ when $\check\rw_t=0$. 
This is because the area process $\check\rwa_k$ is monotone 
on excursions away from 0. 
Furthermore, due to the oscillatory 
nature of the sawtooth path $\check\cP_n$, it suffices to check 
at even
times $t$ when $\check\rw_t=0$. 
Indeed, $\check Y_t$ can only increase from $0$ to $1$ at odd times. 
Therefore, a positive excursion of $\check Y_t$ 
is always preceded by an even time $t$ at which $\check Y_t=0$. 
Therefore, $\check Y_t$ attains its minimum at 
an even time $t$ for which $\check Y_t=0$. 

As such, even times $t\in[0,2n]$  
when $\cZ_t$ occurs will play a special role. 
Roughly speaking, these times will play the role of, 
what we will call, {\it ``almost'' 
renewal times}. 
The reason for the ``almost'' is that the remainder of the walk 
will be shorter than $n$. However, as we will see, such times
most likely occur only very close to the beginning and end 
of the walk. 

Let 
\begin{equation}\label{E_rho}
\rho=\P(\check A_\tau =0),
\quad\quad
\tau 
= \inf\{t: \check Y_t=0,\: \check A_t\le 0\}.
\end{equation}
Note that if $\check Y_{2k}=0$ and $\check A_{2k}< 0$,
then condition \eqref{E_maj1} fails. 
Informally, $\rho$ is the probability that 
when $\sum_{i=1}^k s_i\le {k\choose2}$ first occurs 
at a time when it crosses the sawtooth path,
it is only because $\sum_{i=1}^k s_i={k\choose2}$.

In proving \cref{T_C}, we will obtain 
the following result. 

\begin{thm}
\label{T_asySn_B}
As $n\to\infty$, 
\begin{equation}\label{E_asySn_B}
\frac{n^{5/2}}{4^n}S_n
\to 
\frac{1}{1-\rho}
\frac{\sqrt{3}}{2\pi^{3/2}}
\int_0^\infty
\E[\exp(-6x^3\cB^2)]
\sqrt{1+1/x}\: dx.
\end{equation}
\end{thm}

A key ingredient in the proof of this result
is an asymptotic formula (see \cref{P_KeyF} below), 
relating the asymptotics of 
$S_n$ to a ratio of harmonic moments of two random variables
$N_n$ and $M_n$, which are 
related to the ``almost'' renewal times discussed above. 
We note that \cref{T_asySn_B} follows, using this together with 
\cref{P_numer,P_MntoM}. 

More specifically, the integral is related to 
$\E[N_n^{-1}\mid \cZ_n]$, where $N_n$ 
counts the number even times $0\le t\le 2n$ when $\sY_t=0$. 
Informally, these are potential, ``almost'' renewal times. 
The factor $1/(1-\rho)$, on the other hand, 
is related to $\E[M_n^{-1}\mid \cS_n]$, 
where $M_n$ counts the number of ``almost'' renewal times.
As discussed, ``almost'' renewal times 
are most likely to occur
only very close to the ends of the walks. 
The number of occurrences on either end are approximately geometric,  
and roughly independent, resulting in a negative binomial in the limit. 
In fact, $1-\rho$ is the harmonic moment
$\E(1/X)$ of a shifted (taking values in $x=1,2,\ldots$)
negative binomial random variable $X$, with parameters $r=2$
and $p=1-\rho$.

\subsection{The constant for polymers}
\label{sec:polymer}

The proof of  \cref{cor:polymers} is similar to the proof 
of \cref{T_asySn_B}, the difference being that the role 
of $\check Y$ is played by $Y$. As such, we need to replace 
$\rho$ in \eqref{E_rho} with 
\[
\rho'=\P(A_{\tau'} =0),
\quad\quad
\tau' 
= \inf\{t: Y_t=0,\:  A_t\le 0\},
\]
where $A$ is the area process of $Y$. 
Using \cref{C_AiryInt}, we find that, as $n\to\infty$, 
\[
n^{1/2} \phi_n 
\to 
\frac{1}{2}\sqrt{\frac{\pi}{6}}
\frac{1}{1-\rho'}.
\]

\subsection{Via infinite divisibility}
\label{S_CviaRT}

\cref{T_C} was proved recently in \cite{Kol23}
using limit theory for infinitely divisible sequences.  
We will discuss this briefly, so that 
we can compare these two equivalent descriptions of $C$. 

Let $\bfs$ be a score sequence. 
As discussed, the $\uparrow,\rightarrow$
walk $\cW(\bfs)$ takes $\rightarrow$ steps
at times $s_1+1,\ldots,s_n+n$. 
Note that these times
form a subset of $\{0,1,\ldots,2n-1\}$ that sums to $n^2$. 
We recall that Erd{\H{o}}s, Ginzburg and Ziv \cite{EGZ61} showed that {\it any}
subset of $2n-1$ integers contains a subset of size $n$ that sums to 
a multiple of $n$. When the integers are consecutive, 
von Sterneck (see, e.g., Bachmann \cite{Bac02}) showed 
that the number $V_n$ of such subsets satisfies 
\begin{equation}\label{E_N}
V_n
=\sum_{k=1}^{n}\frac{(-1)^{n+d}}{2n}{2d\choose d},
\quad\quad d={\rm gcd}(n,k). 
\end{equation}

Claesson, Dukes, Frankl\'in and 
Stef\'ansson \cite{CDFS23}
showed that $V_n$ is the ``log transform'' of $S_n$. That is, 
\[
\sum_{n=0}^\infty S_nx^n=
\exp\left(\sum_{k=1}^\infty\frac{V_k}{k}x^k\right).
\]
In other words, the associated generating funtions
satisfy $V(x)=x\frac{d}{dx}\log S(x)$. 
This is used in \cite{Kol23} to observe 
that $p_n=e^{-\lambda}S_n/4^n$,  $n\ge0$, 
is an infinitely divisible probability distribution, where
\[
\lambda=\sum_{k=1}^\infty\frac{V_k}{k4^k}. 
\]
Then, applying the limit theorems 
of 
Hawkes and Jenkins \cite{HJ78}
and 
Embrechts and Hawkes \cite{EH82} 
(based on the analysis of 
Chover Ney and Wainger \cite{CNW73}), 
the asymptotics of $S_n$ are obtained as follows. 

\begin{thm}
[\hspace{1sp}\cite{Kol23}]
\label{T_asySn_N}
As $n\to\infty$, 
\begin{equation}\label{E_asySn_N}
\frac{n^{5/2}}{4^n}S_n
\to \frac{e^\lambda}{2\sqrt{\pi}}.
\end{equation}
\end{thm}

For historical interest, 
we note that 
Kleitman's remarks, appended to \cite[p.\ 166]{Mos71}, 
include intuition which resembles 
the main ideas in the bijective proof in 
\cite[Lemmas 8--10]{CDFS23}. 
Therefore, in retrospect, it appears that all that was
missing at the time (1968) of \cite{Mos71}
towards a proof of Moser's conjecture was
the connection with infinite divisibility (1978)
and von Sterneck's formulas (early 1900s). 
Without these connections, Winston and Kleitman \cite{WK83}
tried 
(1983)
to access these asymptotics 
from quite a different route, and it is this route
which we will follow to its completion in the current work. 

The constant $e^\lambda$ 
has a probabilistic interpretation, and connection with the 	
``almost'' renewal times
discussed in \cref{S_CviaRW} above. 
The sequence $S_n$ is a renewal sequence (see, e.g., 
Feller \cite{Fel68}). This means that the generating function 
satisfies $S(x)=1/(1-S_1(x))$. In this instance, $S_1(x)$ is the generating 
function for {\it irreducible} score sequences, that strictly satisfy \eqref{E_maj1}, 
with equality attained only in \eqref{E_maj2}.
Indeed, any score sequence has a natural decomposition into
irreducible parts, which correspond to  
strongly connected components in the tournament. 
Using 
the ``reverse renewal theorem'' in 
Alexander and Berger \cite{AB16} 
(cf.\ \cite{CNW73}),  
it is shown in \cite{Kol23}
that the number $\cI_n$ of irreducible parts in a uniformly  random score sequence 
converges in distribution 
to a shifted negative binomial random variable with parameters
$r=2$ and $p=e^{-\lambda}$. 

In this work, we will give an alternative proof of the fact that 
$\cI_n$ converges to a negative binomial, and one that 
provides an explanation in terms of random walks, 
as discussed 
at the end of \cref{S_CviaRW} above. 
See \cref{P_MntoM} below.

\subsection{The integral}

Hence $\rho=1-e^{-\lambda}$, 
where $\rho$ is as defined in \eqref{E_rho}.
Therefore, comparing \eqref{E_asySn_B}
and \eqref{E_asySn_N}, 
we obtain \cref{C_AiryInt}.

\section{The correct order}

In this section, we introduce some of the main tools  
which will be used throughout this work. 
In \cref{S_Theta}, we give 
a short proof of \cref{T_order}, 
which identifies the order of $S_n=\Theta(4^n/n^{5/2})$,
as proved in \cite{WK83,KP00}.

\subsection{A local limit theorem}
\label{S_ADL14}

A key ingredient in our proofs is the 
following local limit theorem is 
proved in 
\cite[Proposition 2.1]{ADL14} (cf.\ \cite{CD08}). 

Let $\cV_n$ denote the set of all possible values $(y,a)$ 
for random walk $Y_{2n}$ and its area
$A_{2n}=\sum_{i=1}^{2n} Y_i$. 
In particular, we require that $y\equiv 0$
and $a\equiv n$ mod 2.

\begin{lemma}[Aurzada et al.\ \cite{ADL14}]
\label{L_ADL14}
As $n\to\infty$, 
\[
\max_{(y,a)\in \cV_n}\:
\left| n^2
\P(\rw_{2n}=y,\: \rwa_{2n}=a) 
- \phi(y,a)
\right|
\to0, 
\]
where
\[
\phi(y,a)=
\frac{\sqrt3}{\pi}
\exp\left(
-\frac{y^2}{n}
+\frac{3ay}{2n^2}
-\frac{3a^2}{4n^3}
\right). 
\]
\end{lemma}

We note that, in particular, it follows that 
\begin{equation}\label{E_ZnOrder}
\P(\cZ_n)
=\P(Y_{2n}=0,\: A_{2n}=-n)
\sim \frac{\sqrt3}{\pi}\frac{1}{n^2}. 
\end{equation}
Also, since ${2n\choose n}\sim 4^n/\sqrt{\pi n}$,  
\begin{equation}\label{eq:prob_area_0}
\P( \check\rwa_{2n}=0\mid \check\rw_{2n}=0 )
\sim \sqrt{\pi n}\:
\P(\cZ_n)
\sim 
\sqrt{\frac{3}{\pi}}
\frac{1}{n^{3/2}}.
\end{equation}

\subsection{A cycle lemma}
\label{S_cycle}

We will also use the following analogue of the classical cycle lemma 
(see, e.g., \cite{DM47,DZ90})
for random processes. 

\begin{lemma}\label{L_cycle}
Suppose that $(X_1,\dots,X_n)$ is a random variable in $\R^n$ 
whose law is invariant under cyclic shifts. 
That is, suppose that, for any $j\in [n]$, 
\[
(X_1,\dots,X_n)\overset{d}{=}(X_{1+j},\dots ,X_{n+j }), 
\]
where the indices on the right hand side are to be 
interpreted cyclically, modulo $n$. 
Also suppose that $\sum_{i=1}^nX_i=0$ almost surely.  
Then, 
\[
\P\left(\sum_{i=1}^{k}X_i\ge 0,
\all
1\le k\le n-1\right)\ge 1/n
\]
and 
\[
\P\left(\sum_{i=1}^{k}X_i>0,
\all
1\le k\le n-1\right)\le 1/n.
\]
Equality is attained in both statements 
if and only if $\sum_{i=1}^k X_i$ has a unique 
global minimum almost surely for $1\le k\le n-1$. 
\end{lemma}

\begin{proof}
Note that $\sum_{i=1}^k X_{i+j}\ge 0$ for all $1\le k\le n-1$ 
if and only if $j$ is a time at which $\sum_{i=1}^{j} X_{i}$ 
attains its global minimum. 
Moreover, $\sum_{i=1}^k X_{i+j}> 0$ for all $1\le k\le n-1$ 
if and only if $j$ is the unique time 
at which $\sum_{i=1}^{j} X_{i}$ attains its global minimum. 
    
Let $\cJ$ be the set of indices $j$ 
at which $\sum_{i=1}^{j} X_{i}$ attains its global minimum 
and let $J$ be a uniformly random element in $[n]$,  
so that 
\[
(X_1,\dots,X_n)\overset{d}{=}(X_{1+J},\dots ,X_{n+J}).
\]
    
For the first statement, observe that
\begin{align*}
&\P\left(\sum_{i=1}^{k}X_i\ge0,\all 1\le k\le n \right)\\
&\quad=\P\left(\sum_{i=1}^{k}X_{i+J}\ge0,\all 1\le k\le n \right)\\
&\quad=\E\left[\P\left(\sum_{i=1}^{k}X_{i+J}\ge0,\all 1\le k\le n 
\bigmid X_1,\dots,X_n\right)\right]\\
&\quad= \E[\P(J\in \cJ \mid X_1,\dots,X_n)]
=\frac{1}{n}\E|\cJ|.
\end{align*}
We see that $|\cJ|\ge 1$ with equality 
if and only  $\sum_{i=1}^k X_i$ has a unique 
global minimum almost surely for $1\le k\le n-1$. 

Similarly, for the second statement, we see that 
\begin{align*}
\P\left(\sum_{i=1}^{k}X_i>0,\all1\le k\le n \right)
&= \E[\P(J\in \cJ,\: |\mathcal{J}|=1 \mid X_1,\dots,X_n)]\\
&=\frac{1}{n}\P(|\cJ|=1),
\end{align*}
and that $\P\left(|\cJ|=1\right)\le 1$ with equality, once again,  
if and only  $\sum_{i=1}^k X_i$ has a unique 
global minimum almost surely for $1\le k\le n-1$.
\end{proof}

\subsection{Up to constants}
\label{S_Theta}

Combining \cref{P_SnBridge,L_ADL14,L_cycle}, 
we obtain the following short, probabilistic  proof 
that $S_n=\Theta(4^n/n^{5/2})$. 

As discussed in \cref{S_CviaRW}, the even times $0\le t\le 2n$ when  
$\check A_t=\check Y_t=0$ will play a key role. 
We let 
\[
0=\tau_0<\tau_1<\cdots<\tau_{N_n}=2n
\]
denote the sequence of even times
that $\check Y_t=0$. 
For $1\le k\le N_n$, we let 
\[
\check X_k
=\sum_{s=\tau_{k-1}+1}^{\tau_k}\check Y_s
\]
denote the increment of the area
process $\check A_k$ between times
$\tau_{k-1}$ and $\tau_k$, and 
\[
\check S_k=\sum_{\ell=1}^k \check X_\ell 
\]
the partial sum of these increments. 
Note that 
$\check A_1,\ldots,\check A_{2n}\ge0$ 
if and only if
$\check S_1,\ldots,\check S_{N_n}\ge0$.
We let 
\[
0=\xi_0<\xi_1<\dots<\xi_\nexa=N_n
\] 
denote the sequence of times 
$1\le k\le N_n$ that $\check S_k=0$. 

In the following proof, we will use the fact that 
\begin{equation}\label{E_E1/Nn}
\E[1/N_n]=\Theta(n^{-1/2}), 
\end{equation}
which follows as a 
trivial consequence of 
\cref{P_numer} below.

\begin{proof}[Proof of \cref{T_order}]
By \cref{P_SnBridge}, it suffices to show that 
\[
\P(\cA_n, \: \check A_{2n}=0
\mid \check\rw_{2n}=0 )
=\Theta(n^{-2}).
\]
To this end, by \eqref{eq:prob_area_0}, 
it remains to show that 
\[
\P(\cA_n \mid \cZ_n)
=\Theta(n^{-1/2}).\]
By our observations above, 
\[
\P(\cA_n \mid \cZ_n)
=\P(\check{S}_1,\dots,\check{S}_{N_n} \ge 0 
\mid \cZ_n).
\]
We will estimate this probability using \cref{L_cycle}.

{\bf Lower bound.} 
On the event $\cZ_n$, 
the increments of $\check{S}$ satisfy the conditions of 
\cref{L_cycle}. Indeed, they sum to $0$ and are invariant under cyclic shifts. 
Therefore, by the first bound in \cref{L_cycle}, for any $N$,
\[
\P(\check{S}_1,\dots,\check{S}_{N_n} \ge 0 
\mid \cZ_n,\: N_n=N)
\ge 1/N.
\]
Therefore, by \eqref{E_E1/Nn}, 
\[
\P(\check{S}_1,\dots,\check{S}_{N_n} \ge 0 
\mid \cZ_n )
\ge \E[1/N_n]=\Theta(n^{-1/2}), 
\]
which concludes the proof of the lower bound. 

{\bf Upper bound.} 
For the upper bound, 
we will use the second bound in \cref{L_cycle}. 
However, this bound involves
a strict inequality. As such, we will perform a small trick. 
First, we apply \cref{L_cycle} and \eqref{E_E1/Nn} to obtain
\[
\P(\check{S}_1,\dots,\check{S}_{N_{n+3}-1} > 0 
\mid \cZ_{n+3} )
\le  \E[1/N_{n+3}]
=\Theta(n^{-1/2}). 
\]

To relate this to the probability of interest, 
we consider a special set of paths.  
Let $\cE_{n+3}$ be the event 
$(\check Y_1,\dots, \check Y_{2n+6})$ starts 
with 
\[
(\check Y_0,\check Y_1,\check Y_2)=(0,1,0)
\]
and ends with 
\[
(\check Y_{2n+2},
\ldots, 
\check Y_{2n+6})
=(0,0,-1,0,0).
\]
These conditions ensure that the 
first excursion has area $1$ and that the last excursion 
has area $-1$. See \cref{F_WKtrick}. 
Then, we see that
\begin{align*}
&
\P(\check{S}_1,\dots,\check{S}_{N_{n+3}-1} > 0 
\mid  \cZ_{n+3} )\\
&\quad \ge \P(\cE_{n+3},\: \check{S}_1,\dots,\check{S}_{N_{n+3}-1} > 0 
\mid \cZ_{n+3} ) \\
&\quad = \P(\cE_{n+3} \mid \cZ_{n+3})
\P(\check{S}_1,\dots,\check{S}_{N_{n+3}-1} > 0 
\mid \cE_{n+3},\:  \cZ_{n+3}). 
\end{align*}
If $\cE_{n+3}$
and 
$\cZ_{n+3}$ hold, 
then 
$\check{S}_1,\dots,\check{S}_{N_{n+3}-1} > 0$ holds 
if and only if 
the area process of $(\check{Y}_{i+2})_{0\le i\le 2n}$ 
remains non-negative. 
Moreover, the process 
$(\check{Y}_{i+2})_{0\le i\le 2n}$, on the events $\cE_{n+3}$
and 
$\cZ_{n+3}$, 
has the same law as $(\check{Y}_{i})_{0\le i\le 2n}$, on the event 
$\cZ_{n}$. 
Therefore, 
\[
\P(\check{S}_1,\dots,\check{S}_{N_{n+3}-1} > 0 
\mid \cE_{n+3},\:  \cZ_{n+3})
=\P(\cA_n 
\mid \cZ_{n}). 
\]

Finally, we see that 
\[
\P(\cE_{n+3}\mid \cZ_{n+3})
=
2^{-6}\P(\cZ_{n})/\P(\cZ_{n+3})
\] 
is bounded away from $0$ by \cref{L_ADL14}. 
This implies the upper bound. 
\end{proof}

\begin{figure}[h]
\centering
\includegraphics[scale=1]{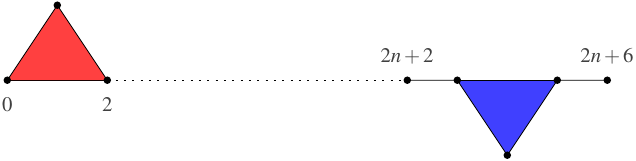}
\caption{The trick in the proof of \cref{T_order}
is to ask for a score sequence
of length $n+3$
of the form $(1,\ldots,n-2,n-2)$. 
}
\label{F_WKtrick}
\end{figure}

\section{The key formula}
\label{S_KeyF}

The key to proving 
\cref{T_C} is the following
asymptotic equivalence. 

\begin{prop}
\label{P_KeyF}
As $n\to\infty$, 
\begin{equation}\label{E_KeyF}
\frac{n^{5/2}}{4^n}S_n
\sim 
\frac{3^{1/2}}{\pi}
\frac{n^{1/2}\E[N_n^{-1}\mid \cZ_n]}
{\E[M_n^{-1}\mid \cS_n]}.
\end{equation}
\end{prop}

The rest of this section is devoted to the proof of 
this result. The numerator and 
denominator will be 
investigated 
separately in \cref{S_numer,S_denom} below. 
As we will see (\cref{P_numer,P_MntoM} below), 
the right hand side converges to a constant. 

\begin{proof}
In light of \cref{P_SnBridge} and \eqref{eq:prob_area_0}, 
we focus on calculating 
\[
\P(\cA_n\mid \cZ_n),
\]
which, as discussed above, is equivalent to 
\[
\P(\check S_1,\ldots,\check S_{N_n}\ge0
\mid \cZ_n).
\]
We will use
\cref{L_cycle} to calculate this probability. 
However, since the minimum of the process $\check S$
is not necessarily unique, we will first perturb its increments by small, random 
amounts. 
To be precise, let $\eps_1,\ldots,\eps_{N_n}$ be IID 
samples from $[-1/2n,1/2n]$ under the condition that $\sum_\ell \eps_\ell=0$. 
We put $\check X_k^* = \check X_k+\eps_k$. Then 
\[
\check S_k^* = \check S_k+\sum_{\ell=1}^k \eps_\ell. 
\]
For the modified process $\check S^*$, \cref{L_cycle} applies. Indeed, 
the increments of $\check S^*$ are invariant under cyclic shifts
and sum to $0$ and the minimum of $\check S^*$ is unique almost surely. 
As such, 
\[
\P(\check S_1^*,\ldots,\check S_{N_n}^*\ge0
\mid \cZ_n,\: N_n=N)
=1/N,
\]
and so 
\begin{equation}\label{E_Nn}
\P(\check S_1^*,\ldots,\check S_{N_n}^*\ge0
\mid \cZ_n)
=\E[N_n^{-1}\mid \cZ_n].
\end{equation}

Next, we note that the perturbations are sufficiently small 
so that 
$|\check S_k-\check S_k^*|<1$ for all $k$. 
In particular, $\check S_k^*\ge0$ implies 
$\check S_k\ge0$. Therefore, 
\begin{align}
&\P(\check S_1^*,\ldots,\check S_{N_n}^*\ge0
\mid \cZ_n)\nonumber\\
&
=\P(\check S_1^*,\ldots,\check S_{N_n}^*\ge0,\: 
\check S_1,\ldots,\check S_{N_n}\ge0 
\mid \cZ_n)\nonumber\\
& 
=\P(\check S_1^*,\ldots,\check S_{N_n}^*\ge0
\mid \cZ_n,\: \check S_1,\ldots,\check S_{N_n}\ge 0)
\P(\check S_1,\ldots,\check S_{N_n}\ge 0
\mid \cZ_n)
\label{eq:relation_perturbed_unperturbed}.
\end{align}
By \eqref{E_Nn}, it remains to find an expression 
for the first factor on the right hand side, as 
then we can calculate the probability of interest.

Recalling that $|\check S_k-\check S_k^*|<1$, we see that 
$\check S_k^*>0$ if $\check S_k>0$.
Therefore, on the event $\check S_1,\ldots,\check S_{N_n}\ge 0$, 
 the condition $\check S_k^*\ge 0$ can only be violated
if $\check S_k=0$ and, in this case, the condition is violated if and only if 
$\sum_{\ell=1}^k\eps_i<0$. Hence, recalling that  
\[
0=\xi_0<\xi_1<\dots<\xi_\nexa=N_n
\]
are the times $k$ 
when $\check S_k=0$, we find that 
\begin{align*}
&\P(\check S_1^*,\ldots,\check S_{N_n}^*\ge0
\mid \cZ_n,\: \check S_1,\ldots,\check S_{N_n}\ge 0)\\
&\quad =
\P\left(\sum_{\ell=1}^{\xi_k}\eps_\ell\ge0,\all 1\le k\le M_n
\bigmid \cZ_n,\: \check S_1,\ldots,\check S_{N_n}\ge 0\right).
\end{align*}

Finally, we note that, on the event  
$\cZ_n$ and $\check S_1,\ldots,\check S_{N_n}\ge 0$, 
\[
\left(\sum_{\ell=\xi_0+1}^{\xi_1}\eps_\ell,
\sum_{\ell=\xi_1+1}^{\xi_2}\eps_\ell,
\ldots,\sum_{\ell=\xi_{M_n-1}+1}^{\xi_{M_n}}\eps_\ell\right)
\]
is invariant in law under cyclic shifts, sums to $0$, and has
a unique minimum almost surely. Therefore, by \cref{L_cycle}, 
\[
\P(\check S_1^*,\ldots,\check S_{N_n}^*\ge0
\mid \cZ_n, \: \check S_1,\ldots,\check S_{N_n}\ge 0,\: M_n=M)
=1/M,
\]
and so 
\begin{align}\label{E_Mn}
&\P(\check S_1^*,\ldots,\check S_{N_n}^*\ge0
\mid \cZ_n,\: \check S_1,\ldots,\check S_{N_n}\ge 0)\nonumber\\
&\quad=\E[M_n^{-1}
\mid \cZ_n,\: \check S_1,\ldots,\check S_{N_n}\ge 0].
\end{align}
Combining 
\eqref{eq:prob_area_0} and \eqref{E_Nn}--\eqref{E_Mn}, 
together with \cref{P_SnBridge}, 
the result follows. 
\end{proof}

\subsection{Graphical sequences}
\label{S_GraphicalSeq}

We recall that a sequence 
${\bf d}=(d_1\le \ldots\le d_n)\in\Z^n$ 
is {\it graphical} if it can be realized 
as the degree sequence of a graph. 
As discussed in Dahl and Flatberg \cite{DF05}, 
the classical conditions by 
Erd\H{o}s and Gallai \cite{EG60}
can be rephrased in the language of 
majorization. Specifically, such a ${\bf d}$ 
is graphical if and only if 
the sum of its entries is even and
${\bf d}'\in\Z^\kappa$ is {\it weakly majorized} 
by ${\bf d}^*\in\Z^\kappa$, 
where
${\bf d}'$ has entires $d_i'=d_i+1$, 
${\bf d}^*$ has entires
$d^*_i=\#\{j:d_j\ge i\}$, 
and $\kappa=\max\{i:d_i\ge i\}$. In other words, 
\[
\sum_{i=1}^k (d_i+1)\le \sum_{i=1}^k d_i^*,\quad\quad 1\le k\le \kappa.
\]
The difference between weak and usual majorization, is that there is no condition 
on the total sum, as in \eqref{E_maj2} above for score sequences.

The scheme used to obtain \eqref{E_KeyF} 
is inspired by the methods in \cite{BDGJS22}, 
which asymptotically enumerated graphical sequences 
by approximating the probability that a 
lazy simple random walk 
$(Y'_1,\dots,Y'_{2n})$ 
satisfies $\cA_n$ and $\cZ'_n=\{Y'_{2n}=0\}$. 
In this context, the role of \cref{L_cycle} is 
played instead by a result in Burns \cite{Bur07} for random processes,  
whose laws are invariant 
under permutations and sign flips.

\section{The numerator}
\label{S_numer}

Recall that $\cA = 2^{3/2}\cB$  is the Airy area measure.
In this section, we prove the following result, which expresses
the numerator in \eqref{E_KeyF}
as an integral involving the Laplace transform
of $\cB^2$. 

As in \cref{S_Theta}, we let 
\[
0=\tau_0<\tau_1<\cdots<\tau_{N_n}=2n
\]
denote the even times $t$ 
for which $\check\rw_t=0$. 

\begin{prop}\label{P_numer}
As $n\to\infty$, 
\[
n^{1/2}\E[\nex^{-1}\mid \cZ_n]
\to 
\frac{1}{2\pi^{1/2}}
\int_0^\infty
\E[\exp(-6x^3\cB^2)]
\sqrt{1+1/x}\: dx.
\]
\end{prop}

The first step 
in our proof of this result 
is the following observation. 

\begin{lemma}\label{L_tau1}
We have that 
\[ 
\E[\nex^{-1} \mid \cZ_n]
=(2n)^{-1}
\E[\tau_1 \mid \cZ_n].
\]
\end{lemma}

\begin{proof}
By linearity, 
\begin{equation}\label{E_tautele}
2n
\E[\nex^{-1} \mid \cZ_n]
=\E\left[\nex^{-1}\sum_{i=1}^\nex (\tau_i-\tau_{i-1}) 
\bigmid \cZ_n\right].
\end{equation}
Next, note that 
\[
\E[\tau_1 \mid \cZ_n]
=\E\left[
\E[\tau_1-\tau_0 \mid N_n,\: \cZ_n]
\mid \cZ_n\right].
\] 
On the event $\cZ_n$, the intervals 
$\tau_i-\tau_{i-1}$ between even times that $\sY$ hits $0$
are exchangeable. As such, the right hand side above is equal to 
\[
\E\left[
\E\left[
\nex^{-1}\sum_{i=1}^\nex (\tau_i-\tau_{i-1})
\bigmid N_n,\: \cZ_n\right]
\bigmid \cZ_n\right],
\]
which reduces to the right hand side in \eqref{E_tautele},
and thereby concludes the proof. 
\end{proof}

With this result at hand, we turn to the proof
of the main result of this section. 

\begin{proof}[Proof of \cref{P_numer}]
We will show that 
\begin{align*}
&n^{1/2}\E[\nex^{-1}\mid \cZ_n]\\
&\quad\to 
\frac{1}{2\pi^{1/2}}
\int_0^1  
\E\left[\exp\left(- \frac{3}{4} \frac{x^3}{(1-x)^3}\cA^2\right)\right]
x^{-1/2}(1-x)^{-2}dx,
\end{align*}
where
$\cA = 2^{3/2}\cB$ is the Airy area measure.
The result then follows by 
a simple change of variables.

By \eqref{E_ZnOrder}, 
\[
\E[\tau_1 \mid \cZ_n]
\sim 
\frac{\pi n^2}{\sqrt3}
\E[\tau_1 \1_{\{\rw_{2n}=0,\: \rwa_{2n}=-n\}}].
\]
Therefore, by \cref{L_tau1}, it remains to show that 
\begin{align*}
&n^{3/2}
\E[\tau_1 \1_{\{\rw_{2n}=0,\: \rwa_{2n}=-n\}}]\\
&\quad \to \frac{\sqrt3}{\pi^{3/2}}
\int_0^1  
\E\left[\exp\left(- \frac{3}{4} \frac{x^3}{(1-x)^3}\cA^2\right)\right]
x^{-1/2}(1-x)^{-2}dx. 
\end{align*}
We observe that
\begin{align*}
\E[\tau_1 \1_{\{\rw_{2n}=0,\: \rwa_{2n}=-n\}}]
&=\E[\tau_1 \P(\rw_{2n}=0,\: \rwa_{2n}=-n \mid \tau_1,\: \rwa_{\tau_1})]\\
&=\E[\tau_1 \P(\rw'_{2n-\tau_1}=0,\: \rwa'_{2n-\tau_1}=-n-\rwa_{\tau_1})],
\end{align*}
where $(\rw',\rwa')$ is an independent copy of $(\rw,\rwa)$. 
Directly, 
\begin{align}
&\E[\tau_1 \P(\rw'_{2n-\tau_1}=0,\: \rwa'_{2n-\tau_1}=-n-\rwa_{\tau_1})]\nonumber\\
& =2\sum_{t=1}^{n} \sum_{a=-t^2}^{t^2}
t\P(\tau_1=2t,\: \rwa_{\tau_1}=a)\label{E_3regions}
\P(\rw'_{2(n-t)}=0,\: \rwa'_{2(n-t)}=-n-a).
\end{align}
To estimate this sum, we fix $2/3<\beta<1$ and 
$n'=n-n^{1/4}$, 
and then evaluate over the following three regions: 
\begin{enumerate}
\item $t\in [1, n^\beta]$,  
\item $t\in (n^\beta, n']$,  
\item $t\in (n', n]$. 
\end{enumerate}
(For ease of notation, we write $n^\beta$ and $n'$,
rather than $\lfloor n^\beta\rfloor$ and $\lfloor n'\rfloor$.)
First, we will show that the first and third sums are $o(n^{-3/2})$. 
Then we will find the limit of the second sum, rescaled by $n^{3/2}$. 

{\bf First region.}
By \cref{L_ADL14}, there is a constant $c$ such that 
\begin{align*}
\P(\rw'_{2(n-t)}=0,\: \rwa'_{2(n-t)}=-n-a)
\le 
cn^{-2}
\end{align*}
for all $a$ and $t\le n^\beta$. 
Therefore, 
\begin{align*}
&\sum_{t=1}^{n^\beta} \sum_{a=-t^2}^{t^2}t\P(\tau_1=2t,\: \rwa_{\tau_1}=a)  
\P(\rw'_{2(n-t)}=0,\: \rwa'_{2(n-t)}=-n-a)\\
&\quad\le cn^{-2}\sum_{t=1}^{n^\beta}t\P(\tau_1=2t).
\end{align*}
Next, we recall
(see, e.g., Feller \cite[Sec.\ III.7]{Fel68}) that
\begin{equation}\label{E_tau1}
\P(\tau_1=2t)
=\frac{2}{(2t-1)4^t}\binom{2t-1}{t}
\sim \frac{1}{2\sqrt{\pi}t^{3/2}},
\end{equation}
and so 
\[
\sum_{t=1}^{n^\beta} t \P(\tau_1=2t)=O(n^{\beta/2}).
\]
As such, the contribution to \eqref{E_3regions}
from the first region is $O(n^{\beta/2-2})=o(n^{-3/2})$,
since $\beta<1$. 

{\bf Third region.} Next, we claim that, in fact, 
\[
\sum_{t=n'}^n \sum_{a=-t^2}^{t^2}t\P(\tau_1=2t,\: \rwa_{\tau_1}=a)  
\P(\rw'_{2(n-t)}=0,\: \rwa'_{2(n-t)}=-n-a)=0.
\]
Indeed, note that $n-t\le  n^{1/4}$ for $n'\le t\le n$. 
As such, $\rwa'_{2(n-t)}\in [-\sqrt {n}, \sqrt{n}]$ deterministically, and so 
the second factor is $0$ for $a\not\in [-n-\sqrt {n}, -n+\sqrt{n}]$. 
On the other hand, $|\rwa_{\tau_1}|\ge \tau_1-2$ deterministically, 
so the first factor is $0$
for all other $a\in [-n-\sqrt {n}, -n+\sqrt{n}]$, and the claim follows. 

{\bf Second region.} Finally, we turn our attention to the sum over the 
second region, 
which, as we will see, is on the order $n^{3/2}$.
Note that, by \eqref{E_tau1}, 
\begin{align*}
&\sum_{t=n^\beta}^{n'} \sum_{a=-t^2}^{t^2}t\P(\tau_1=2t,\: \rwa_{\tau_1}=a)  
\P(\rw'_{2(n-t)}=0,\: \rwa'_{2(n-t)}=-n-a)\\
&= 
\sum_{t=n^\beta}^{n'} 
\frac{2t\binom{2t-1}{t}}{(2t-1)4^t}
\sum_{a=-t^2}^{t^2}\P(\rwa_{\tau_1}=a\mid \tau_1=2t)  
\P(\rw'_{2(n-t)}=0,\: \rwa'_{2(n-t)}=-n-a).
\end{align*}

Let $\cV_m'=\{a:(0,a)\in\cV_m\}$ be the set of all possible values 
for the area process, if the random walk 
hits $0$ at time $t=2m$. 
Fix $\eps>0$.  Then, \cref{L_ADL14}, 
\[
\max_{a\in \cV_m'}\:
\left| 
m^2\P\left(\rw'_{2m}=0,\: \rwa'_{2m}=a\right) 
- \frac{\sqrt{3}}{\pi} \exp\left(-\frac{3a^2}{4m^3}\right)
\right|<\eps, 
\]
for all $m\ge M$, for some sufficiently large $M$. 
Then, for $n-n'=n^{1/4}\ge M$, 
\begin{align*}
&\left|\sum_{t=n^\beta}^{n'} 
\frac{2t\binom{2t-1}{t}}{(2t-1)4^t}
\sum_{a=-t^2}^{t^2}\P(\rwa_{\tau_1}=a\mid \tau_1=2t)  
\P(\rw'_{2(n-t)}=0,\: \rwa'_{2(n-t)}=-n-a)\right.\\
&\left. -\frac{\sqrt{3}}{\pi} 
\sum_{t=n^\beta}^{n'} \frac{2t\binom{2t-1}{t}}{(2t-1)(n-t)^24^t} 
\sum_{a\in \cV_t'}\P(\rwa_{\tau_1}=a \mid \tau_1=2t)
\exp\left[-\frac{3(a+n)^2}{4(n-t)^3}\right]\right|\\
&\quad\le \epsilon 
\sum_{t=n^\beta}^{n'} 
\frac{2t}{(2t-1)(n-t)^2 4^t}\binom{2t-1}{t}.
\end{align*}

Let $E_t$ be the area of a uniform Bernoulli excursion with length $2t$
(that has positive or negative sign, both with probability $1/2$).
Then 
\[
\sum_{a\in \cV_t'}\P(\rwa_{\tau_1}=a \mid \tau_1=2t)
\exp\left(-  \frac{3}{4}\frac{(a+n)^2}{(n-t)^3}\right)
=
\E\left[\exp\left(-\frac{3}{4}\frac{(E_t+n)^2}{(n-t)^3}\right)\right].
\]
Takács \cite[Theorem 3]{Tak91}
showed that 
$t^{-3/2}|E_t|\overset{d}{\to}\cA$. 
We claim that, as $n\to \infty$, 
\begin{equation}\label{E_EtvsA}
\sup_{n^\beta \le t\le n'}
\left|
\E\left[\exp\left(-\frac{3}{4}\frac{(E_t+n)^2}{(n-t)^3}\right)\right]
-\E\left[\exp\left(-\frac{3}{4}\frac{t^3\cA^2}{(n-t)^3}\right)\right]
\right|
\to 0.
\end{equation}
Since $\beta>2/3$, we have that 
$n^2=o(t^3)$.
As such, 
$(E_t+n)^2/t^3\overset{d}{\to}\cA^2$.  
We will work on a probability space 
where this convergence holds almost surely. 
Then, for any $\xi,\zeta>0$, 
\begin{align*}
&\sup_{n^\alpha \le t\le n'}
\left|
\E\left[\exp\left(-\frac{3}{4}\frac{t^3}{(n-t)^3}\frac{(E_t+n)^2}{t^3}\right)\right]
-\E\left[\exp\left(-\frac{3}{4}\frac{t^3\cA^2}{(n-t)^3}\right)\right]
\right|\\
&\le \sup_{n^\alpha \le t\le n'} 
\E\left[ \sup_{r>0}
\left|\exp\left(- r \frac{(E_t+n)^2}{t^3}\right) - \exp(- r \cA^2)\right| \right]\\
&\le \sup_{n^\alpha \le t\le  n'} 
\left[ 
\P\left(\frac{(E_t+n)^2}{t^3}\le \xi\right)
+\P\left(\left| \frac{(E_t+n)^2}{t^3}- \cA^2 \right|\ge  \zeta\right)
\right]\\
&\quad\quad 
+\P(\cA^2\le \xi)
+\sup_{r>0}\sup_{\substack{x,y> \xi\\ |x-y|<\zeta} }
|e^{- r x} - e^{- ry}|.
\end{align*}
Let $\eps>0$ be given. 
The functions $e^{-rx}$, for $r>0$, are uniformly Lipschitz, so 
the final term is less than $\eps/4$ for small enough $\zeta>0$. 
By the almost sure convergence and the fact that $\cA$ puts no mass on $0$, 
it follows that, for any $\xi,\zeta>0$, all three probabilities in the last expression are
less than $\eps/4$. Therefore the left hand side in 
\eqref{E_EtvsA} is bounded by $\eps$ for all large $n$, and the claim follows. 

Altogether, we find that 
\begin{align*}
&2\sum_{t=n^\beta}^{n'} \sum_{a=-t^2}^{t^2}t\P(\tau_1=2t,\: \rwa_{\tau_1}=a)  
\P(\rw'_{2(n-t)}=0,\: \rwa'_{2(n-t)}=-n-a)\\
&\quad \sim
\frac{\sqrt{3}}{\pi^{3/2}} 
\sum_{t=n^\beta}^{n'} 
\frac{1}{\sqrt{t}(n-t)^2}
\E\left[\exp\left(-\frac{3}{4}\frac{t^3}{(n-t)^3}\cA^2 \right)\right]\\
&\quad\sim
\frac{\sqrt{3}}{\pi^{3/2}}
\frac{1}{n^{3/2}}
\int_0^1  
\E\left[\exp\left(- \frac{3}{4} \frac{x^3}{(1-x)^3}\cA^2\right)\right]
x^{-1/2}(1-x)^{-2}dx.
\end{align*}
Since the contribution from the other regions is $o(n^{3/2})$, we find 
by \eqref{E_3regions} 
that 
\begin{align*}
&n^{3/2}
\E[\tau_1 \1_{\{\rw_{2n}=0,\: \rwa_{2n}=-n\}}]\\ 
&\quad 
\to \frac{\sqrt{3}}{\pi^{3/2}}
\int_0^1  
\E\left[\exp\left(- \frac{3}{4} \frac{x^3}{(1-x)^3}\cA^2\right)\right]
x^{-1/2}(1-x)^{-2}dx,
\end{align*}
as required. 
\end{proof}

\section{The denominator}
\label{S_denom}

\subsection{Near the boundary}

The denominator in \eqref{P_KeyF} involves $M_n$, the number 
of $k\le n$ 
such that $\cZ_k$ occurs. 
Our first lemma shows that, with high probability, 
all such times are close to the start or end of the walk. 

\begin{lemma}\label{lem:onlyhitnearboundary}
Suppose that $1\ll m(n)\ll n$.  Then  
\[ 
\P(\exists k\in(m,n-m):\cZ_k
\mid
\cS_n)
\to 0  
\]
as $n\to\infty$.
\end{lemma}

\begin{proof}
By \cref{T_order,P_SnBridge}, 
it follows that 
\[
p_n
=\P(\cS_n)
=\Theta(n^{-5/2}).
\] 
Therefore, for some $C>0$, 
\[
\P(\cZ_k
\mid 
\cS_n)
=\frac{p_kp_{n-k}}{p_n}\le C\left(\frac{n}{k(n-k)}\right)^{5/2},
\]
for all $k\in(0,n)$. 
Hence, 
\begin{align*}
&\P(\exists k\in(m,n-m):\cZ_k
\mid
\cS_n)\\
&\quad\le C\sum_{k=m}^{n-m} \left(\frac{n}{k(n-k)}\right)^{5/2} 
\le 2C\sum_{k=m}^{\lceil n/2\rceil } \left(\frac{n}{k(n-k)}\right)^{5/2}\\
&\quad\le 2^{7/2}C \sum_{k=m}^{\infty} k^{-5/2}
= O(m^{-3/2})=o(1),
\end{align*}
as claimed. 
\end{proof}

\subsection{``Almost'' renewal times}

Next, we will prove the following technical lemma, 
which is the key to making the idea of 
``almost'' renewal times more precise. 

\begin{lemma}\label{lem:ratio_qk}
Let $q_n=\P(\cA_n\mid\cZ_n)$. 
Then $q_{n-m}/q_n\to 1$, uniformly in $m\le \log n$.
\end{lemma}

We note that, by \cref{T_order}, \cref{P_SnBridge} and \eqref{E_ZnOrder}, 
\begin{equation}\label{E_qnOrder}
q_{n}=\Theta(n^{-1/2}). 
\end{equation}

The proof of \cref{lem:ratio_qk} is given at the end of this section, 
after a number of preliminary results are proved, see 
\cref{lem:changeconditioning,lem:controlR-N,lem:no_bad_things,L_almost} 
below. 

The main idea is to compare the ``long'' process on $[2n]$ 
with a ``shorter'' process on $[2(n-m)]$. 
For $m\ll b\ll n$,  
we will show that the law of the first and last $n-b$ steps in the 
two processes are asymptotically indistinguishable. 
We will then show that, in both process, if the area process 
is non-negative during the first and last $n-b$ steps, 
then it is unlikely to go negative in between.  

We will start by studying the Radon--Nikodym derivative 
between the first and last $n-b$ steps in the two processes. 
Before stating our first lemma, we will set some notation. 

Let $Z$ be a function on the space of walks,   
of length $2n$ and starting at $0$, 
that only depends on 
the first and last $n-b$ increments.  
To be precise, for $n>b$ and $\rw=(\rw_1, \dots, \rw_{2n})$, 
we let $\rw^\leftarrow_k=\rw_{2n-k}-\rw_{2n}$ 
denote the time-reversed process of $\rw$. 
We assume that $Z$ is measurable with respect to 
\[
\Sigma_{n-b}
=\sigma(\rw_1,\dots,\rw_{n-b},\:
\rw^\leftarrow_1, \dots, \rw^\leftarrow_{n-b}).
\]
We note that 
\[
\Sigma_{n-b}
=\sigma(\Delta\rw_1,\dots,\Delta\rw_{n-b},\:
\Delta\rw_{n+b+1}, \dots, \Delta \rw_{2n}), 
\]
where $\Delta \rw_k=\rw_{k}-\rw_{k-1}$.  
Finally, we let  
\[
\rwa^\leftarrow_{n-b}=\sum_{i=1}^{n-b}\rw^\leftarrow_i
\] 
be the area under the last $n-b$ increments of $\rw$,
assuming that $Y_{2n}=0$. 
Note that $\rwa^\leftarrow_{n-b}$ is measurable 
with respect to $\Sigma_{n-b}$.

\begin{lemma}\label{lem:changeconditioning}
Let 
\begin{equation}\label{E_varphi}
\varphi_\ell(y,a)=\P(\check Y_{2\ell}=y,\: \check A_{2\ell}=a).
\end{equation}
Let $\E_n$ denote expectation 
with respect to the law of $\check Y$ on $[2n]$,  
Then, for any $m<b<n$, we have that 
\[
\E_n\left[Z\mid\cZ_n\right]
=\E_{n-m}\left[Z 
\frac{\varphi_b(\gamma,\alpha)}
{\varphi_{b-m}(\gamma,\alpha')}
\frac{\varphi_{n-m}(0,0)}
{\varphi_n(0,0)}
\bigmid \cZ_{n-m}
\right],
\]
where
\begin{align}
&\gamma=\sY^\leftarrow_{n-b}-\sY_{n-b}, \label{E_gamma}\\
&\alpha=-\sA_{n-b}-\sA^\leftarrow_{n-b}-2b\sY_{n-b},\label{E_alpha}\\
&\alpha'=-\sA_{n-b}-\sA^\leftarrow_{n-b}-2(b-m)\sY_{n-b}.\label{E_alpha'} 
\end{align}
\end{lemma}

\begin{proof}
Suppose that $b'<n'$ and that 
$Z'$ only depends on $\Sigma_{b'}$. 
Then for $\E_{n'}$, the expectation with respect to the law of $\sY$ on $[2n']$, 
we have that 
\begin{align*}
\E_{n'}[Z'\mid \cZ_{n'}]
&=\frac{\E_{n'}[Z'\1(\cZ_{n'})] }{\P_{n'}(\cZ_{n'})}\\
&= \frac{\E_{n'}\left[\E_{n'}[Z'\1(\cZ_{n'})\mid \Sigma_{b'} ]\right]}{\P_{n'}(\cZ_{n'})}\\
&= \frac{\E_{n'}[Z' \P_{n'}(\cZ_{n'} \mid \Sigma_{b'})]}{\P_{n'}(\cZ_{n'})}.
\end{align*} 
    
We can rewrite the event  
$\cZ_{n'}$
in terms of random variables that 
are measurable with respect to $\Sigma_{b'}$ 
or independent of $\Sigma_{b'}$, 
by separating the increments in the first and last 
$b'$ steps from the remaining increments 
as follows: 
\[ 
\left\{\sY_{2n'-b'}-\sY_{b'}
=\sY^{\leftarrow}_{b'}-\sY_{b'},\: 
\sum_{i=1}^{2(n'-b')}(\sY_{b'+i}-\sY_{b'})
=-\sA_{b'}-\sA^{\leftarrow}_{b'}-2(n'-b')\sY_{b'} \right\}.
\]
Indeed, since $\sY_{2n'-b'}-\sY^{\leftarrow}_{b'}=\rw_{2n'}$, 
the first equality ensures that $\sY_{2n'}=0$, 
and the second equality clearly corresponds to 
$\sA_{2n'}=0$.  

Next, to rewrite 
$\P_{n'}(\cZ_{n'} \mid \Sigma_{b'})$, 
we note that  
$(\sY_{b'+i}-\sY_{b'})$  
and $(\sY_{i})$
have the same law 
conditional on $\Sigma_{b'}$. 
Hence 
\[
\E_{n'}\left[Z'
\frac{\varphi_{n'-b'}(\sY^{\leftarrow}_{b'}-\sY_{b'},
-\sA_{b'}-\sA^{\leftarrow}_{b'}-2(n'-b')\sY_{b'})}
{\varphi_{n'}(0,0)}\right]
=\E_{n'}[Z'\mid \cZ_{n'}].
\]
Finally, applying this equality with 
\[
Z'=Z 
\frac{\varphi_b(\gamma,\alpha)}{\varphi_{b-m}(\gamma,\alpha')}
\frac{\varphi_{n-m}(0,0)}{\varphi_n(0,0)}, 
\] 
$n'=n-m$ and $b'=n-b$,
we obtain the result. 
\end{proof}

In our application of \cref{lem:changeconditioning}, we will 
take $b=\lfloor n^{9/20}\rfloor$ and $m\le \log n$. 
The next two lemmas find a region that, with high probability, 
contains the values of $\gamma$, $\alpha$ and $\alpha'$ in 
\eqref{E_gamma}--\eqref{E_alpha'}, and in which 
the Radon--Nikodym derivative is close to $1$. 

We recall that $\omega(1)$
denotes a function that $\to\infty$
as $n\to\infty$. 

\begin{lemma}\label{lem:no_bad_things}
Let 
\[
\cR=\cR_\gamma\cap\cR_\alpha\cap\cR_{\alpha'}, 
\]
where 
\begin{align*}
&\cR_\gamma=\{|\gamma|\le n^{1/4}\}, \\
&\cR_\alpha=\{|\alpha|\le n^{29/40}\}, \\
&\cR_{\alpha'}=\{|\alpha-\alpha'|\le n^{3/5}\}. 
\end{align*}
Then, for $b=\lfloor n^{9/20}\rfloor$ and $m\le \log n$, 
we have that, as $n\to\infty$, 
\begin{align*}
&\P_{n}(\cR\mid \cZ_n)=1-n^{-\omega(1)},\\
&\P_{n-m}(\cR\mid \cZ_{n-m})=1-n^{-\omega(1)}.
\end{align*}
\end{lemma}

\begin{proof}
We will only show the first statement, 
as the second statement follows similarly. 
Note that $|\gamma|=|\sY_{n-b}-\sY_{n-b}|$. 
Also, observe that, on the event 
$\cZ_n$, 
\[
\sA_n=\sA_{n-b}+\sA^\leftarrow_{n-b}
+\sum_{i=1}^{2b} \sY_{n-b+i}=0. 
\]
Hence 
\[
\alpha=-\sA_{n-b}-\sA^\leftarrow_{n-b}-2b\sY_{n-b}
=\sum_{i=1}^{2b}(\sY_{n-b+i}-\sY_{n-b}).  
\]
Therefore, 
since 
\[
2bn^{1/4}= O(n^{7/10})\ll n^{29/40},
\] 
for all large $n$, 
\[
\cR_\gamma^c\cup(\cR_\alpha^c\cap \cZ_n)
\subset \left\{\max_{0\le i \le 2b} |\sY_{n-b+i}-\sY_{n-b}|>n^{1/4} \right\}. 
\]
By the reflection principle, and 
Hoeffding's inequality, 
this event has probability at most  
$\exp[-\Omega(n^{1/2}/b)]$. 
Therefore, since $b=O(n^{9/20})$, 
it follows that 
\[
\P(\cR_\gamma^c\cup(\cR_\alpha^c\cap \cZ_n))
\le 
\exp[-\Omega(n^{1/20})].
\]

Next, note that $|\alpha-\alpha'|=2m|\sY_{n-b}|$. 
As such, Hoeffding's inequality implies  
\[
\P(\cR_{\alpha'}^c)\le \exp[-\Omega(n^{1/5}/\log^2n)].
\]

Altogether, we find that $\P(\cR^c\cap \cZ_n)=n^{-\omega(1)}$. 
Since $\P(\cZ_n)=\Theta(n^{-2})$ by 
\eqref{E_ZnOrder}, it follows that 
$\P(\cR^c\mid \cZ_n)=n^{-\omega(1)}$, as claimed. 
\end{proof}

Next, we show that, on the event $\cR$, 
the Radon--Nikodym derivative is close to $1$. 

\begin{lemma}\label{lem:controlR-N}
Suppose that 
$b=\lfloor n^{9/20}\rfloor$
and $m\le \log n$. 
Then
\[
\frac{\varphi_b(y,a)}{\varphi_{b-m}(y,a')}
\frac{\varphi_{n-m}(0,0)}{\varphi_n(0,0)}
\to 1, 
\]
as $n\to \infty$, 
uniformly over 
$|y|\le n^{1/4}$, $|a|\le n^{29/40}$ and $|a-a'|\le n^{3/5}$, 
for which $\varphi_b(y,a)>0$ and $\varphi_{b-m}(0,0)>0$. 
\end{lemma}

\begin{proof}
By \cref{L_ADL14}, it follows that, 
for some constants $c_1,\dots,c_4$ 
and for some $\eps_\ell$, 
with $\eps_\ell\to 0$ as $\ell\to\infty$,
we have that 
\[ 
\varphi_\ell(y,a)
=\frac{c_1}{\ell^2}
\exp\left[
c_2\frac{y^2}{\ell}
+c_3\frac{ay}{\ell^2}
+c_4\frac{a^2}{\ell^3}
+\eps_\ell
\right],
\]
for $a$, $\ell$ and $y$ such that 
$\varphi_\ell(y,a)>0$.

In fact, the parity conditions for
$\varphi_b(y,a)>0$ and $\varphi_b(y,a')>0$ 
(and similarly for $\varphi_{n-m}(0,0)$ and $\varphi_n(0,0)$)
are equivalent, so we can ignore these conditions from now on
(since $\E_n$ does not put any mass on values 
which do not satisfy these conditions). 
Note that 
the parity condition for $\varphi_b(y,a)>0$ 
requires that $y\equiv 0$  and $a \equiv 0$ mod 2.
Similarly, for $\varphi_b(y,a)>0$, 
we require that $y\equiv 0$  and $a' \equiv 0$ mod 2. 
However, since $a-a'\equiv 0$ mod 2, 
the conditions are indeed equivalent.     

Next, suppose that  
$a_0/a\to0$ and 
$\ell_0/\ell\to0$. 
Then, after some straightforward algebraic manipulations, 
it can be seen that 
\begin{align*} 
&\frac{\varphi_\ell(y,a)}{\varphi_{\ell-\ell_0}(y,a-a_0)}\\
&\quad =
(1-\ell_0/\ell)^2
\exp\left[
O\left(\frac{y^2\ell_0+ya_0}{\ell^2}
+\frac{ya\ell_0+aa_0}{\ell^3}
+\frac{a^2\ell_0}{\ell^4}
\right)
+o(1)
\right].
\end{align*}
Finally, we observe that, if $\ell=\Omega(n^{9/20})$ 
and $\ell_0\le \log n$, 
then 
this ratio converges to $1$, uniformly in all 
$|y|\le n^{1/4}$, 
$|a| \le n^{29/40}$ and 
$|a_0| \le n^{3/5}$. 
\end{proof}

By \cref{lem:changeconditioning,lem:controlR-N,lem:no_bad_things}
it follows that, for $k\le\log n$ and $b=\lfloor n^{9/20}\rfloor$, 
that the law of the first and last $n-b$ steps of the ``long'' process on $[2n]$ 
and the ``shorter'' process on $[2(n-b)]$ are asymptotically indistinguishable. 
This implies the following result, which is very close to \cref{lem:ratio_qk}. 
In fact, the only difference is that, in the following, 
we do not consider the sign of the area process 
between times $n-b$ and $n+b$ in the ``long'' process 
and between times $n-b$ and $n-2m+b$ in ``shorter'' process. 

Let 
\[
\cA_{n'}^{n,b}
=\{\check\rwa_1,\dots,\check\rwa_{n-b}\ge 0,\:
\check\rwa_{2n'-n+b},\dots,\check\rwa_{2n'}\ge 0
\}. 
\]
  
\begin{lemma}\label{L_almost}
Suppose that $b=\lfloor n^{9/20}\rfloor$. 
Then 
\[
\P_n(\cA_{n}^{n,b} \mid\cZ_n) 
=(1+o(1))\P_{n-m}(\cA_{n-m}^{n,b} \mid\cZ_{n-m}),
\]
where $o(1)\to0$, uniformly in 
$m\le \log n$.
\end{lemma}
    
\begin{proof}
Let $\underline{\delta}_n$ and 
$\overline{\delta}_n$ denote the minimum/maximum values of 
\[
\frac{\varphi_b(y,a)}{\varphi_{b-m}(y,a')}
\frac{\varphi_{n-m}(0,0)}{\varphi_n(0,0)},
\]
taken over the set of 
all $m\le \log n$, 
$|y|\le n^{1/4}$, $|a|\le n^{29/40}$ and $|a-a'|\le n^{3/5}$,  
for which $\varphi_b(y,a)>0$ and $\varphi_{b-m}(0,0)>0$. 
By \cref{lem:controlR-N}, $\underline{\delta}_n, \overline{\delta}_n\to 1$,
uniformly. 

Next, we apply \cref{lem:changeconditioning} with 
\[
Z=\1(\sA_1,\ldots,\sA_{n-b}\ge0,\:
\sA_1^\leftarrow,\ldots,\sA_{n-b}^\leftarrow\ge0,\: \cR). 
\]
We find that 
\begin{align*}
\P_n(\cA_{n}^{n,b} \mid\cZ_n)
&\le \P_n(\cA_{n}^{n,b},\: \cR \mid\cZ_n) + \P_n(\cR^c \mid\cZ_n)\\
&\le \overline{\delta}_n \P_{n-m}(\cA_{n-m}^{n,b},\: \cR \mid\cZ_{n-m}) 
+ \P_n(\cR^c \mid\cZ_n)\\
&\le \overline{\delta}_n \P_{n-m}(\cA_{n-m}^{n,b} \mid\cZ_{n-m}) 
+ \P_n(\cR^c \mid\cZ_n).
\end{align*}   
Likewise, 
\begin{align*}
\P_n(\cA_{n}^{n,b} \mid\cZ_n)
&\ge \P_n(\cA_{n}^{n,b},\: \cR \mid\cZ_n)\\
&\ge \underline{\delta}_n \P_{n-m}(\cA_{n-m}^{n,b},\: \cR \mid\cZ_{n-m})\\
&\ge \underline{\delta}_n [\P_{n-m}(\cA_{n-m}^{n,b}\mid\cZ_{n-m})
-\P_{n-m}(\cR^c \mid\cZ_{n-m})].
\end{align*} 

Finally, note that $\P_n(\cA_{n}^{n,b} \mid\cZ_n)\ge q_n$. 
Therefore, the result 
follows by \cref{lem:no_bad_things}
and \eqref{E_qnOrder}.
\end{proof}

We are ready to prove the main result of this section. 
As discussed above, it remains only to show that 
the sign of the area process 
is unlikely to change
during the intermediate times not covered by 
\cref{L_almost}. 

\begin{proof}[Proof of  \cref{lem:ratio_qk}]
By \cref{L_almost}, it suffices to show that, 
as $n\to\infty$, 
\begin{equation}\label{L_middlen}
\frac{\P_n(\cA_{n}^{n,b} \mid\cZ_n)}{q_n}\to1
\end{equation}
and
\begin{equation}\label{L_middlenm}
\frac{\P_{n-m}(\cA_{n-m}^{n,b} \mid\cZ_{n-m})}{q_{n-m}}\to1
\end{equation}
uniformly in $m\le \log(n)$. 

We will show \eqref{L_middlen}, and \eqref{L_middlenm}
follows similarly. By \eqref{E_qnOrder}, it suffices to show 
\[ 
\P(
\exists k \in (n-b, n+b]: \sgn(\sA_{k-1})\neq \sgn(\sA_k)
\mid \cZ_n)
\ll n^{-1/2}.
\]

Note that 
\begin{align*}
&\{\exists k \in (n-b, n+b]: \sgn(\sA_{k-1})\neq \sgn(\sA_k)\}\\
&\quad\subset 
\{ |\sA_{n-b}|<2b n^{21/40} \}
\cup 
\left\{ \max_{1\le i\le 2b }|\sY_{n-b+i}|>n^{21/40}   \right\}. 
\end{align*}   
By Hoeffding's inequality,
\[
\P\left(\max_{1\le i\le 2b }|\sY_{n-b+i}|>n^{21/40}\right)
=n^{-\omega(1)}.
\]
So we need only show that 
\[ 
\P(|\sA_{n-b}|<2b n^{21/40} \mid \cZ_n)
\ll
n^{-1/2}.
\]
By \cref{lem:changeconditioning}, 
this probability 
can be expressed as the expected value of 
\begin{align*}
&\1\{|\sA_{n-b}|<2b n^{21/40}\} \\
&\quad\times\frac{\P(\sY'_{n+b}=-\sY_{n-b},\: 
\sA'_{n+b}=-\sA_{n-b}-(n+b)\sY_{n-b}\mid \sY_{n-b}, \sA_{n-b}) }
{\P(\sY'_{2n}=\sA'_{2n}=0)}, 
\end{align*}
where $\sY'$ is an independent copy of $\sY$, 
and $\sA'$ its area process.

However, by \cref{L_ADL14}, there is a constant $c$ such that 
\[
\P(\sY_{n+b}'=y,\: \sA_{n+b}'=a)
=O(n^{-2}). 
\] 
By \eqref{E_ZnOrder}, 
\[
\P(\sY'_{2n}=\sA'_{2n}=0)=\Theta(n^{-2}). 
\]
It follows that, for some $c>0$, 
\[
\P(|\sA_{n-b}|<2b n^{21/40} \mid \cZ_n)
\le c
\P(|\sA_{n-b}|<2b n^{21/40}).
\]

Note that any value of $a$ can be realised by $O(\sqrt{a})$ 
values of $y $. Therefore, 
again by \cref{L_ADL14}, 
\[
\P(|\sA_{n-b}|=a)
=  O(\sqrt{a}/n^{2}), 
\] 
and so
\[
\P(|\sA_{n-b}|<2b n^{21/40})
=O((bn^{21/40})^{3/2}/n^{2})
=O(n^{-43/80})
\ll n^{-1/2},
\] 
as required. 
\end{proof}

\subsection{Negative binomial limit}

Finally, in this section, we will formalize the idea of 
the ``almost'' renewal times $t$, when $\sY_{2t}=\sA_{2t}=0$. 
Recall that the reason for ``almost'' is that after such a time, 
there is then less time remaining for the rest of the trajectory of the
walk. However, by \cref{lem:onlyhitnearboundary}, 
with high probability, such times $t$
only occur very close to the start and end of the walk. 
As such, the probability of an ``almost'' renewal time 
\begin{equation}\label{E_rho_n}
\rho_n
=\P\left(\exists k\le \log n: \cZ_k \mid \cS_n\right)
\end{equation}
does not depend on $n$
asymptotically.

Recall $\rho$ as defined in \eqref{E_rho}.

\begin{lemma}\label{lem:conv_rho}
As $n\to \infty$, 
we have that 
$\rho_n\to \rho$.
\end{lemma}

\begin{proof}
Let 
\begin{equation}\label{E_zeta1}
\zeta_1
=\min\{k>0:\sY_{2k}=0,\: \sA_{2k}\le 0\}.
\end{equation} 
For any $k\le \log n$,
\begin{align*}
\P(\cA_n\mid\cZ_n,\: \zeta_1=k)
&=\P(\cA_k\mid\cZ_n,\: \zeta_1=k)
\P(\sA_{2k+1},\ldots,\sA_n\ge0\mid\cZ_k,\:\cZ_n)\\
&=\P(\sA_{2\zeta_1}=0 \mid\cZ_n,\: \zeta_1=k)q_{n-k}.
\end{align*}
Therefore, setting $m=\lfloor \log n\rfloor$, 
we find that 
\begin{align*}
&\P(\zeta_1\le m \mid \cS_n)\\
&=\sum_{k=1}^{m}\frac{
\P(\zeta_1=k \mid \cZ_n) 
\P( \cA_n\mid \cZ_n,\: \zeta_1=k )}
{\P(\cA_n \mid \cZ_n )}\\
&=\sum_{k=1}^{m} 
\P(\zeta_1=k \mid \cZ_n) 
\P(\sA_{2\zeta_1}=0 \mid\cZ_n,\: \zeta_1=k)\frac{q_{n-k}}{q_n}.
\end{align*}
Applying \cref{lem:ratio_qk}, it follows that 
\[
\P(\zeta_1\le m \mid \cS_n)
= \P(\sA_{2\zeta_1}=0,\: \zeta_1\le m \mid \cZ_n)
+o(1). 
\]

Next, we will show that the conditioning on $\cZ_n$ 
has a negligible effect on the probability 
of the event
$\{\sA_{2\zeta_1}=0,\: \zeta_1\le m\}$.  
Intuitively, note that this event is measurable 
with respect to the first $2m$ steps of the process, 
and that, on this scale, the process 
will hardly notice that it is tied down at the end. 

To make this rigorous, we use ideas which 
are similar to the proof of \cref{lem:ratio_qk}. 
Specifically, note that 
\[
\frac{\varphi_{n-m}(y, a)}{\varphi_n(0,0)}\to 1, 
\]
uniformly in $y\le n^{1/4}$ and $a\le n^{5/4}$, 
where is $\varphi_\ell(y,a)$ as defined in \eqref{E_varphi}. 
Let
\[
\cG
=\{ |\sY_{2m}|, |\sA_{2m}|\le n^{1/4}\}, 
\]
which occurs deterministically for all large $n$. 
Hence, for all such $n$, 
\begin{align*}
\P(\sA_{2\zeta_1}=0,\: \zeta_1\le m \mid \cZ_n)
&=\P(\sA_{2\zeta_1}=0,\: \zeta_1\le m,\:\cG)+o(1)\\
&=\P(\sA_{2\zeta_1}=0,\: \zeta_1\le m)+o(1).
\end{align*}

Finally, note that 
$\sA_{2\zeta_1}$ does not depend on $n$. 
Therefore, 
\[
\P(\sA_{2\zeta_1}=0,\: \zeta_1\le m)
\to \P(\sA_{2\zeta_1}=0,\: \zeta_1<\infty)
=\rho, 
\]
since $\zeta_1<\infty$ almost surely, 
which concludes the proof. 
\end{proof}

Finally, we obtain the asymptotic
distribution of the number of ``almost'' renewal times $M_n$, 
that is, the number of events $\cZ_k=0$, $k\le n$. 

\begin{prop}
\label{P_MntoM}
Let $M$ be a negative binomial random variable 
with parameters $p=1-\rho$ and $r=2$. 
Then, conditional on $\cS_n$, we have that 
$M_n\overset{d}{\to}1+M$, 
as $n\to \infty$.
In particular, $\E[M_n^{-1}\mid \cS_n]\to 1-\rho$. 
\end{prop}

\begin{proof}
Once again, 
we let $m=\lfloor \log n \rfloor$. 
We will use 
\cref{lem:conv_rho} 
to show that, on the event $\cS_n$, 
the random variables 
\begin{align*}
&M^{(1)}_n=\# \left\{ 0<k\le m :\cZ_k\right\},\\
&M^{(2)}_n=\# \left\{ n-m\le k<n:\cZ_k \right\},
\end{align*}
jointly converge to  
independent geometric random variables, with parameters
$p=1-\rho$. This implies that $M^{(1)}_n+M^{(2)}_n\overset{d}{\to}M$, 
and the result then follows by \cref{lem:onlyhitnearboundary}. 

First, note that, by \cref{lem:conv_rho}  and symmetry, 
\[
\P(M^{(i)}_n\ge 1\mid \cS_n)
=\rho_n\to \rho,\quad\quad i=1,2. 
\]
Next, we claim that, for any $m_1,m_2 \geq 0$, 
\begin{equation}\label{E_OneMore}
\P(M^{(1)}_n\ge m_1+1 
\mid M^{(1)}_n\ge m_1,\: 
M^{(2)}_n\ge m_2,\:  
\cS_n)\to \rho. 
\end{equation}
Then, by symmetry, we find that, for any $m_1,m_2 \ge 0$, 
\[
\P=(M^{(1)}_n\ge m_1,\: M^{(2)}_n \ge m_2 
\mid  \cS_n)\to \rho^{m_1+m_2},
\]
which implies the lemma.

Let $\zeta^{(1)}_j$ 
(resp.\ $\zeta^{(2)}_j$)
denote the
index of the $j$th smallest (resp.\ largest) time $t\in(0,n)$ 
that $\cZ_t$ occurs. In particular, $\zeta_1=\zeta^{(1)}_1$, 
as defined in \eqref{E_zeta1}. 
We note that, conditional on 
\[
\cS_n^{m_1,m_2}(j,k)=
\{M^{(1)}_n\ge m_1,\:
\zeta_{m_1}^{(1)}=j,\: 
M^{(2)}_n\ge m_2,\:
\zeta_{m_2}^{(2)}=k,\: 
\cS_n
\}, 
\]
the process 
$(\sY_{2j+i})_{0\le i \le 2(n-j-k)}$ 
has the same law as  
$(\sY_{i})_{0\le i \le 2(n-j-k)}$, 
conditional on $\cS_{n-j-k}$. 
Therefore, for $j,k=O(\log\log n)$, we have that 
\begin{align*}
\P(M^{(1)}_n\ge m_1+1\mid \cS_n^{m_1,m_2}(j,k))
&=\P(\exists k \in (0,m-j]:\cZ_k\mid \cS_{n-j-k})\\
&=\rho_{n-j-k}+o(1).
\end{align*} 
The $o(1)$ correction accounts for the discrepancy 
between $m$ and $m-j$. Indeed, 
this discrepancy has a negligible $o(1)$ affect, by 
\cref{lem:onlyhitnearboundary}.

Therefore, by \cref{lem:conv_rho}, the above converges
to $\rho$, uniformly in all $j,k=O(\log\log n)$. 
On the other hand, by \cref{lem:onlyhitnearboundary}, 
\[
\P( \zeta_{m_1}^{(1)} >\log\log n \mid 
M^{(1)}_n\ge m_1,\: M^{(2)}_n\ge m_2,\: \cS_n)
\to 0. 
\]
Then, since
\[
\P( M^{(1)}_n\ge m_1,\: M^{(2)}_n\ge m_2
\mid 
\cS_n)
\]
is bounded away from 0, 
\eqref{E_OneMore} follows, as required. 
\end{proof}

\section{The scaling limit}
 
Finally, in this section, we will prove
\cref{T_ScalingLimit}. 

We let 
$(\bfY_t,\bfA_t)$ denote the 
{\it positive Kolmogorov excursion (from zero and back)}  
in \cite{BDW23}.
We will first observe the following scaling limit
for the lattice path associated with a random score sequence. 

\begin{prop}\label{P_EtoY}
Let
${\bf S}^{(n)}=({\bf S}^{(n)}_1\le \cdots\le {\bf S}^{(n)}_n)$
be a uniformly random score 
sequence of length $n$. 
Let 
${\bf Y}^{(n)}
=({\bf Y}^{(n)}_0,{\bf Y}^{(n)}_1,\ldots,{\bf Y}^{(n)}_{2n})$ 
denote the bridge path $\cP({\bf S}^{(n)})$
obtained by rotating the walk ${\bf W}^{(n)}=\cW({\bf S}^{(n)})$ clockwise
by $\pi/4$. 
Then, as $n\to\infty$, we have that
\[
\left((2n)^{-1/2}{\bf Y}^{(n)}_{\lfloor 2nt \rfloor},\: 
0\le t\le 1\right)
\overset{d}{\to}
(\bfY_t,\: 0\le t\le 1), 
\]
where $(\bfY_t,\bfA_t)$ is a Kolmogorov excursion. 
\end{prop}

\begin{proof}

Let $(Y_i,\: 0\le i\le 2n)$ be a simple random walk bridge, 
conditioned on its area process $A_k=\sum_{i=1}^kY_i$ 
satisfying 
\begin{align}
&A_k\ge -\lceil k/2\rceil,\quad\quad 1\le k<2n,\label{E_kol1}\\ 
&A_{2n}=-n. \label{E_kol2} 
\end{align}

As discussed in \cref{S_intro}, by the bijection in \cite[p.\ 212]{WK83}
and a rotation clockwise by $\pi/4$, it follows that score sequences
$\bfs=(s_1,\ldots,s_n)$ 
satisfying \eqref{E_maj1} and \eqref{E_maj2}
are in bijection with bridges $(Y_i,\: 0\le i\le 2n)$
satisfying \eqref{E_kol1} and \eqref{E_kol2}.
Indeed, 
the bijection associates each $\bfs$ to 
a bridge $Y=Y(\bfs)$ with 
increments given by 
\[
Y_i-Y_{i-1}
=
\begin{cases}
-1,&i\in\{s_1+1,\ldots,s_n+n\};\\
+1,&i\in[2n]\setminus \{s_1+1,\ldots,s_n+n\}. 
\end{cases}
\]
As such, 
it suffices to show that 
\begin{equation}\label{E_kol}
((2n)^{-1/2}Y_{\lfloor 2nt \rfloor},\: 0\le t \le 1)
\overset{d}{\to}
(\bfY_t,\: 0\le t \le 1). 
\end{equation}

To show \eqref{E_kol}, we will 
first consider the following 
modified processes. 
Put
\begin{align*}
&A_k'=A_k+k/2+3/4,\\
&Y_k'=Y_k+1/2.
\end{align*}
We claim that 
\begin{align*} 
&\left\{A'_k>0,\all1\le k \le 2n,\: A'_{2n}=3/4,\: Y'_{2n}=1/2\right\}\\
&\quad=\left\{A_k\ge -\lceil k/2\rceil,\all 1\le k \le 2n,\: A_{2n}=-n,\:  Y_{2n}=0\right\}.
\end{align*}
Clearly, $A'_{2n}=3/4$ and $Y'_{2n}=1/2$ if and only if $A_{2n}=-n$ and $Y_{2n}=0$. 
On the other hand, since $A$ can only take integer values, it follows that, 
for integers $\ell\ge1$, 
\begin{align*} 
&\{A_{2\ell}\ge -\ell\}=\{A_{2\ell}>-\ell-3/4\} = \{A'_{2\ell}>0\},\\
&\{A_{2\ell-1}\ge -\ell\}=\{A_{2\ell-1}>-\ell-1/4\} = \{A'_{2\ell-1}>0\},
\end{align*} 
so the claim follows. 
Therefore, by 
\cite[Theorem 2]{BDW23}, 
\[
((2n)^{-1/2} Y'_{\lfloor 2nt \rfloor}, 
(2n)^{-3/2}A'_{\lfloor 2nt \rfloor},\: 0\le t \le 1)
\overset{d}{\to}
(\bfY_t,\bfA_t,\: 0\le t \le 1), 
\]
in the topology of uniform convergence on c\`adl\`ag processes on $[0,1]$. 
Since $|Y'_n-Y_n|$  and $|A'_n-A_n|$
are of smaller order than the scaling factors,
\eqref{E_kol} follows, and this completes the proof. 
\end{proof}

Recall that score sequences 
are in bijection with bridges
with partial areas $\ge0$ and
total area $0$ above the sawtooth path. 
\cref{F_WK,F_WKrot} above demonstrate the 
forward direction of this bijection. Essentially, 
a score sequence
$\bfs=(s_1,\ldots,s_n)$ is drawn as a ``bar graph,''
and then rotated clockwise by $\pi/4$. 
Roughly speaking, we will obtain \cref{T_ScalingLimit}
by reversing this procedure in the continuum. 

While $\bfs=(s_1,\ldots,s_n)$ is of length $n$, its associated 
bridge path ${\bf y}=(y_0,y_1,\ldots,y_{2n})$ is of length $2n$, with $n$ {\it up steps}
and $n$ {\it down steps}. However, only the times
of the down steps are needed to recover $\bfs$. 
More specifically, let $\vartheta_k$ be the time {\it before} the $k$th 
down step. 
Then $s_k=y_{\vartheta_k}+k-1$, where $k-1$
is the height of the $k$th step along the staircase walk. 
Informally, after the path is rotated counterclockwise by $\pi/4$, 
the up steps become vertical lines. These lines are 
``traversed instantly,'' as we read the score sequence
as a ``bar graph'' 
from left to right, so time changes from $[0,2t]$
to $[0,t]$.  See \cref{F_WKinv}. 

\begin{figure}[h]
\centering
\includegraphics[scale=1]{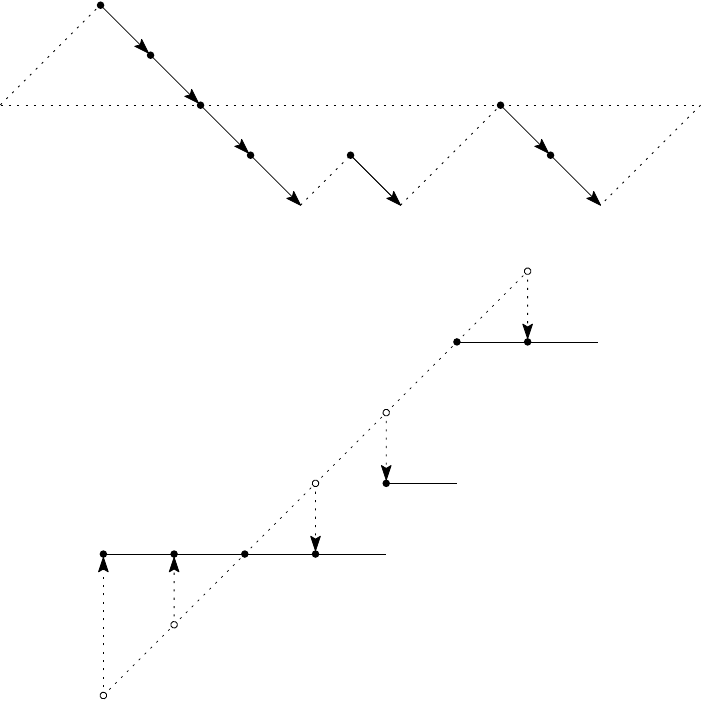}
\caption{
{\it Above:} The bridge path in \cref{F_WKrot}, 
with its values before down steps 
indicated. {\it Below:} The 
bar graph in \cref{F_WKrot} is recovered 
by adding these values to
the staircase walk. 
}
\label{F_WKinv}
\end{figure}

Since, by \cref{P_EtoY}, the excursion 
has a Brownian limit, which changes 
directions constantly, the process 
\[
\left(n^{-1/2}({\bf S}^{(n)}_{\lfloor nt \rfloor}-\lfloor nt \rfloor),\: 0\le t\le 1\right)
\overset{d}{\to}
(\bfY_t,\: 0\le t\le 1) 
\]
will have the same scaling limit 
$\bfY_t$ 
as the bridge path ${\bf Y}^{(n)}$ in \cref{P_EtoY}. 
Note that $(\lfloor nt \rfloor,\: 0\le t\le 1)$ is the
staircase walk.

\begin{proof}[Proof of \cref{T_ScalingLimit}]

Let $(X_i,\: 0\le i\le 2n)$ be a simple random walk, 
so that $X$ conditioned on $\cS_n$ is distributed as 
$\bY^{(n)}$, as in \cref{P_EtoY}. 
Let $\vartheta_k=\vartheta_k(X)$ be the time before the 
$k$th down step  of $X$. 
Formally, we put 
$\vartheta_0=0$ and then, for $1\le k \le n$, 
\[
\vartheta_k=\inf\{\ell>\vartheta_{k-1}: X_\ell >X_{\ell +1}\}.
\]
Then, we see that, for $1\le k \le n$, the corresponding 
walk $\bW^{(n)}$ passes through the points 
$(k-1,\bY^{(n)}_{\vartheta_k}+k-1)$ and $(k,\bY^{(n)}_{\vartheta_k}+k-1)$. 
Therefore, 
\begin{equation}\label{eq:S_and_E}
\bS^{(n)}_k=\bY^{(n)}_{\vartheta_k}+k-1,\quad\quad 1\le k\le n. 
\end{equation}
  
Next, we will show that 
\begin{equation}\label{E_diag}
\left((2n)^{-1}\vartheta_{\lfloor n t\rfloor}(\bY^{(n)}),\: 0\le t\le 1\right)
\overset{d}{\to} (t,\: 0\le t\le 1), 
\end{equation}
in the topology of uniform convergence. Then, the theorem 
follows by
 \cref{P_EtoY}, 
\eqref{eq:S_and_E}
and the fact that $(\bY_t,\:0\le t\le 1)$ is continuous almost surely. 
For $1\le k \le 2n$, let 
\[
\vartheta^{-1}_k
=\vartheta^{-1}_k(X)
=\# \{\ell\le k: X_{\ell-1}>X_{ \ell}\}
\] 
be the number of down steps in the first $k$ steps. 
Then 
$\vartheta_k={\max\{\ell:\vartheta^{-1}_\ell < k\}}$. 
It suffices to show that 
\begin{equation}\label{E_diag_inv}
\left(n^{-1} \vartheta^{-1}_{\lfloor 2nt \rfloor}(\bY^{(n)}),\: 0 \le t \le 1\right)
\overset{d}{\to} 
(t,\: 0\le t \le 1), 
\end{equation}
in the topology of uniform convergence, 
since $(t,\: 0\le t \le 1)$ is strictly increasing. 
Let $B_1,B_2,\dots$ be i.i.d.\ 
$\Bern(1/2)$ random variables. 
Note that 
\[
(\vartheta^{-1}_k(X),\: 1\le k \le 2n)
\overset{d}{=} 
\left( \sum_{i=1}^k B_i,\: 1\le k \le 2n\right).
\]
Therefore, 
applying Chernoff's bound at times $\lfloor \eps k n/2 \rfloor$, 
and then taking a union bound, we find that, 
for some $\alpha>0$,
\[ 
\P\left(
\sup_{1\le k \le 2n  } 
n^{-1}\left| \vartheta^{-1}_k(X)-k/2\right| 
>\epsilon
\right)
\le e^{-\alpha n},
\]
for all large $n$. 
Hence, as $n\to\infty$, 
\[
\P\left(\sup_{1\le k \le 2n  } 
n^{-1}
\left| \vartheta^{-1}_k(\bY^{(n)})-k/2\right| >\epsilon\right)
\le 
\frac{e^{-\alpha n} }{\P(\cS_n)}\to0,
\]
since $\P(\cS_n)=\Theta(n^{-5/2})$. 
This yields  \eqref{E_diag_inv}, and then  
\eqref{E_diag} follows.  
As discussed, this completes the proof. 
\end{proof}

\makeatletter
\renewcommand\@biblabel[1]{#1.}
\makeatother

\providecommand{\bysame}{\leavevmode\hbox to3em{\hrulefill}\thinspace}
\providecommand{\MR}{\relax\ifhmode\unskip\space\fi MR }
\providecommand{\MRhref}[2]{%
  \href{http://www.ams.org/mathscinet-getitem?mr=#1}{#2}
}
\providecommand{\href}[2]{#2}

\end{document}